\documentclass[12pt]{amsart}
\usepackage{amsmath,amssymb,amscd,amsxtra, amsfonts, amsthm, amsbsy}
\usepackage{latexsym}
\usepackage[dvips]{graphics,epsfig}
\usepackage{float}

\newcommand{\nm}[1]{\|{#1}\|}

\newcommand{\ip}[2]{\langle#1,#2\rangle}

\newtheorem{thm}{Theorem}[section]
\newtheorem{lemma}{Lemma}[section]

\newtheorem{cor}{Corollary}[section]
\theoremstyle{remark}
\newtheorem{rem}{Remark}[section]
\theoremstyle{definition}
\newtheorem{deft}{Definition}[section]

\usepackage{setspace}
\DeclareSymbolFont{AMSb}{U}{msb}{m}{n}
\DeclareMathSymbol{\N}{\mathbin}{AMSb}{"4E}
\DeclareMathSymbol{\Z}{\mathbin}{AMSb}{"5A}
\DeclareMathSymbol{\R}{\mathbin}{AMSb}{"52}
\DeclareMathSymbol{\Q}{\mathbin}{AMSb}{"51}
\DeclareMathSymbol{\I}{\mathbin}{AMSb}{"49}
\DeclareMathSymbol{\C}{\mathbin}{AMSb}{"43}

\begin{document}

\title{Orthogonal Polynomials on the Sierpinski Gasket}
\author{Kasso A.~Okoudjou}
\address{Kasso A.~Okoudjou\\
Department of Mathematics\\
University of Maryland\\
College Park, MD 20742 USA}
\email{kasso@math.umd.edu}

\author{Robert S.~Strichartz}
\address{Robert S.~Strichartz\\
Department of Mathematics\\
 Malott Hall\\
Cornell
\newline University, Ithaca, NY 14853-4201, USA}
\email{str@math.cornell.edu}

\author{Elizabeth K.~Tuley}
\address{Elizabeth K.~Tuley\\
Department of Mathematics\\
University of California at Los Angeles\\
Los Angeles, CA 90095 USA}
\email{ektuley@math.ucla.edu}


\date{\today}


\begin{abstract}
The construction of a Laplacian on a class of fractals which includes the Sierpinski gasket ({\bf $SG$}) has given rise to an intensive research on analysis on fractals. For instance, a complete theory of polynomials and power series on $SG$ has been developed by one of us and his coauthors. We build on this body of work to construct certain analogs of  classical orthogonal polynomials (OP) on $SG$.  In particular, we investigate key properties of these OP on $SG$, including a three-term recursion formula and the asymptotics of the coefficients appearing in this recursion. Moreover, we develop numerical tools that allow us to graph a number of these OP. Finally, we use these numerical tools to investigate the structure of the zero and the nodal  sets of these polynomials. 

\end{abstract}

 \pagestyle{myheadings} \thispagestyle{plain}
\markboth{K. A. OKOUDJOU, R. S. STRICHARTZ, AND E. K. TULEY}{ORTHOGONAL POLYNOMIALS ON THE SIERPINSKI GASKET}

\maketitle 

\tableofcontents

\setlength{\parskip}{1ex plus 0.5ex minus 0.3ex}

\section{Introduction}\label{intro}
The polynomials $\pi_{0}(t)=1$, and $\pi_k(t) = t^k + \sum_{l=0}^{k-1}a_{l}t^{l}, k=1,2,\ldots$ defined on the unit interval $I$ are called monic orthogonal polynomials with respect to a weight $w(t)$ if
\begin{equation}
\langle \pi_k, \pi_j \rangle = \int_I \pi_k(t) \pi_j(t) w(t) dt = \|\pi_{k}\|_{2}^{2} \delta_{k,j}\quad k=0, 1, 2, \hdots,
\end{equation}
where $\delta_{k,j}$ denotes the Kronecker $\delta$ sequence. 

When renormalized, these  monic orthogonal polynomials give rise to an orthonormal system of polynomials $\{\tilde{\pi}_{k}\}_{k=0}^{\infty}$, where $\langle \tilde{\pi_k}, \tilde{\pi_j} \rangle = \delta_{j,k}$

The monic OP are constructed by performing the Gram-Schmidt process on the monomials $\{t^k\}_{k=0}^{\infty}$, with respect to the inner product $\langle a, b \rangle = \int_I a(t) b(t) w(t) dt $. Examples of such polynomials include  the Legendre polynomials defined on $I=[-1,1]$, and using the weight function $w(t) = 1$, 

Numerical information about the OP can be obtained via a fundamental recursive equation called the three-term recursion formula.  
For example,  the  values of the polynomials at different points in $I$, and their zeros can be obtained from the three-term recurrence relation. For comparison with our results, we state the general form of this recurrence relation and  refer to \cite{gaut, szego75} for details on OP.  

{\bf Theorem A. \cite[Theorem 1.27]{gaut}} {\it
The orthogonal polynomials $\{\pi_k\}_{k\geq 0}$ satisfy a three-term recurrence relation. More specifically,  $\pi_{-1}(t) = 0, \pi_0(t) = 1,$  and for all $k\geq 0$ we have: 
\begin{equation}\label{classic3tr}
\pi_{k+1}(t) = (t-\alpha_k) \pi_k(t) - \beta_k \pi_{k-1}(t),
\end{equation}
where 
\begin{equation}
\left\{ \begin{array} {r@{\quad = \quad}c@{\quad \textrm{for}\quad}c@{\quad \geq \quad}c}
\alpha_k & \frac{\langle t\pi_k, \pi_k \rangle}{\langle \pi_k, \pi_k \rangle} &  k&0\\
\beta_k & \frac{\langle \pi_k, \pi_k \rangle}{\langle \pi_{k-1}, \pi_{k-1} \rangle}& k&1.\end{array}\right.
\end{equation}
}

Our goal in this paper is to construct and investigate the properties of the analogs of the above OP on $SG$. But first what is a polynomial on a fractal set such as $SG$?

Let $\Delta=\tfrac{d^{2}}{dx^{2}}$ denote the Laplacian on $I=[0,1].$
If we let  $\pi_j(t)= t^j$, for $t \in I=[0,1]$, it is immediate that
$\pi_j$ is a solution to a differential equation
$\Delta^{\ell}\pi_j(t)=0,$ for some $\ell\geq 1$. This observation can
be used to define polynomial on fractals in general and on the
Sierpinski gasket in particular.  In particular, on $SG$, one defines a \emph{polynomial} $P$ to be any solution of $\Delta^{j+1} P=0$ for some $j\geq 0$, where $\Delta$ is the (fractal) Laplacian to be defined below.  Using this analogy, monomials were introduced on the Sierpinski gasket  \cite{dsv, nsty1, bssy2, str06, strus} based on the fractal Laplacian constructed by Kigami \cite{kig}, see also \cite{str06, Str99}. We also note that a theory of polynomials based on a different Laplacian, the so-called energy Laplacian, is developed in \cite{stts10}.

Consequently, our construction of   OP on $SG$ is  based on these
polynomials. In particular, we apply the Gram-Schmidt
orthogonalization algorithm to produce the analogs to the Legendre
polynomials on $SG$.  In the process we construct both a family of
symmetric and antisymmetric OP on $SG$, and the concatenation of the
two families yields a tight frame for a proper closed subspace of
$L^{2}(SG)$. The reason that our construction does not yield a dense
set of OP in $L^{2}(SG)$ is due to the fact that the polynomials are
not dense in this space \cite[Theorem 4.3.6]{kig}, \cite[Section
5.1]{str06}. In addition, we investigate thoroughly a corresponding
three-term recursion formula, paying a particular attention to the
asymptotics of the coefficients involved in this recursion. Our
three-term recurrence is fundamentally different
from~\eqref{classic3tr}, and this is essentially due to the fact that
the product of two polynomials on $SG$ is not a polynomial
\cite{bbst}! In addition, our analog of~\eqref{classic3tr} is very
unstable and thus cannot be used to generate the values of the OP on
$SG$. However, we are able to exploit this three-term recurrence to
recursively express the corresponding OP in terms of a well-known
basis of polynomials on $SG$. This leads to the development of some
numerical tools that allow us to plot graphs of many of these OP. In
addition, we present preliminary results pertaining the zeros of these
OP. In particular, we graph the zero sets of certain of these
OP. These results are mainly experimental but seems to indicate that
these zero sets have some structure. We have not been able to prove
anything along these lines but hope that our experimental results lead
to more investigations on the zero sets of the OP.  Nonetheless, we
have generated many graphs related to the  OP on $SG$ that we did not include in the present paper due to space constraint. The interested reader is referred to the following web site which contains many figures and algorithms that were generated in the course of this research project \cite{webpage}.

 Our paper is organized as follows: In Section~\ref{analysis} we review the basic facts of analysis on fractals needed to state and prove our results. We also prove some new results about the Green's function that may be of independent interest.  Section~\ref{sec2} is devoted to the construction of the OP on $SG$ and to the investigation of the three-term recursion and the asymptotic analysis of its coefficients. We also consider the Jacobi matrix associated with these coefficients.  In Section~\ref{sec3} we carry out the numerical computations that enable us to plot many of the OP we constructed. In addition, we give some numerical evidence concerning the asymptotics of the coefficients involved in the three-term relation. Finally, we give numerical description of the  zero sets of the OP and display some of these zero sets as well as some of their corresponding nodal sets on $SG$.

\section{Polynomials on $SG$}\label{analysis}
\subsection{Preliminaries}

\begin{figure}[h]
\begin{center}
\includegraphics[scale=.20]{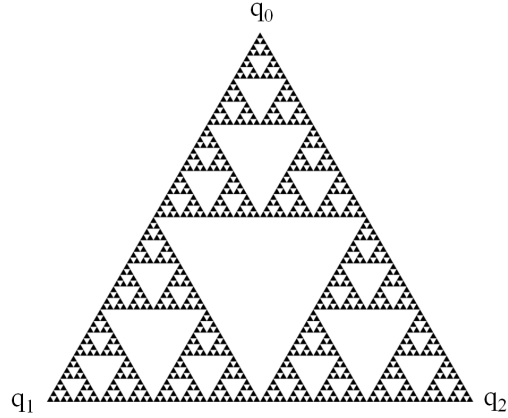}
\caption{The Sierpinski Gasket}
\label{fig:SG}
\end{center}
\end{figure}

For more information on the theory of calculus on fractals see \cite{kig}; more specifically for calculus on the Sierpinski Gasket see \cite{str06, Str99}. However, we collect below some key facts needed in the formulation of our results. 

The Sierpinski Gasket ($SG$), shown in Figure \ref{fig:SG}, is the attractor of iterated function system (IFS) consisting of three contractions in the plane $F_i: \R^2 \to \R^2$, $i=0,1,2$, defined by $F_i = \frac{1}{2}(x-q_i)+q_i$ where $\{q_i\}_{i=0}^2$ are the vertices of an equilateral triangle. $SG$ satisfies the self-similar identity $SG = \cup_{i=0}^2 F_i(SG)$, and the $F_i(SG)$ are called \emph{cells of level 1}.  By iteration we can write $F_w = F_{w_1} \circ F_{w_2} \circ \ldots \circ F_{w_m}$ for $w = (w_1,w_2,\ldots,w_m)$, each $w_j = 0,1,2$.  We call each $w$ a word of length $|w| = m$, and we have the self-similar identity $SG = \cup_{|w|=m}F_w(SG)$.  Each $F_w(SG)$ is called a \emph{cell of level $m$}.  

$SG$ can also be viewed as the limit of a sequence of graphs $\Gamma_m$ (with vertices $V_m$ and edge relation $x \sim_m y$) defined inductively as follows: $\Gamma_0$ is the complete graph on $V_0 = \{q_0,q_1,q_2\}$, and $V_m = \cup_{i=0}^2 F_i V_{m-1}$ with $x \sim_m y$ if and only if x and y belong to the same cell of level $m$. Then $V_* = \cup_{m=1}^\infty V_m$, the set of all \emph{vertices} is the analog of the dyadic points in $[0, 1]$ and is dense in $SG$. We consider $V_0$ the set of boundary points of $SG$, and $V_* \setminus V_0$ is the set of \emph{junction points}. Note that every junction point in $V_m$ has four neighbors in the graph $\Gamma_m$.  

The graph Laplacian $\Delta_m$ is defined by 
\begin{equation}
\Delta_m u(x) = \sum_{x \sim_m y} (u(y) - u(x)) \mbox{\hspace{.2in} for } x\in V_m \setminus V_0
\end{equation}
and the Laplacian $\Delta$ on $SG$ is defined as the renormalized limit
\begin{equation}
\Delta u(x) = \lim_{m \to \infty} \frac{3}{2} 5^m \Delta_m u(x).
\end{equation}

At the three boundary points of $SG$, we have two different derivatives, the normal and tangential derivative, respectively defined by 
\begin{equation}\label{partderi}
\left\{ \begin{array} {r@{\quad = \quad}l}
\partial_n u(q_i) & \lim_{m \to \infty} \left( \frac{5}{3} \right)^m (2u(q_i) -u(F^m_i q_{i+1}) -u(F^m_i q_{i-1}))\\
\partial_T u(q_i) & \lim_{m \to \infty} 5^m(u(F^m_0 q_{i+1}) - u(F^m_0 q_{i-1})),\end{array}\right.
\end{equation} identifying the boundary points $q_i$ and $q_j$ when
$i=j \mod 3$.  

Using the above definitions, the analog of the Gauss-Green formula takes the following form: 
\begin{equation}
\int_{SG} (u\Delta v - v\Delta u)d\mu = \sum_{i=0}^2 ( u(q_i)\partial_n v(q_i) - v(q_i)\partial_n u(q_i)),
\end{equation}
where  $\mu$ is the natural probability measure which assigns weight $3^{-m}$ to each cell of order $m$.

These analytical tools can now be used to solve the boundary value problem  $$-\Delta u = f, \qquad u|_{V_0} = 0, $$  whose solution is given by  
\[u = \int_{SG} G(x,y)f(y)\,d\mu(y). \]
 $G$ is called the Green's function, and is  given by 
\[G(x,y) = \lim_{M \to \infty} G_M(x,y) \mbox{, where } G_M(x,y) = \sum_{m=0}^M \sum_{z'\in V_{m+1}\setminus V_m} g(z,z') \psi_z^{(m+1)}(x) \psi_{z'}^{(m+1)}(y), \]
where $g(z,z')$ is zero when $z$ and $z'$ are not in the same cell of level $m+1$, 
\[g(z,z) = \frac{9}{50} \left(\frac{3}{5}\right)^m \mbox{\hspace{.2in} for } z\in V_{m+1}\setminus V_m\]
\[\mbox{\hspace{.1in} and}\, g(z,z') = \frac{3}{50} \left(\frac{3}{5}\right)^m \mbox{\hspace{.2in} for } z\neq z',\mbox{\hspace{.05in} $z$ and $z'$ in the same cell of level } m+1,\]
and $\psi_z^{(m+1)}$ is the piecewise harmonic function satisfying $\psi_{z}^{(m+1)}(x)=\delta_{zx}$ and vanishing outside the $(m+1)$th cell to which $z$ belongs.

We shall need an exact estimate of $\nm{G}_{L^{2}}$, the $L^2$ norm of the Green function $G$. It was conjectured in \cite{kss00} and proved in \cite{wat02} that $$\nm{G}_{L^{\infty}}=178839/902500\simeq .198.$$ This immediately implies that $\nm{G}_{L^{2}} < \nm{G}_{L^{\infty}} <1$. However, we obtain  exact values for $\nm{G}_{L^{2}}$ and a related quantity in the next result. Though the result seems simple, to our knowledge, it has not appeared anywhere in the literature. 

First, we briefly recall a description of the spectrum of $\Delta$ on $SG$,  that  was given in~\cite{FuSh92} using the
method of spectral decimation introduced in~\cite{RaTo83}. In essence, the spectral
decimation method completely determines the eigenvalues and the eigenfunctions of $\Delta$ on $SG$
from the eigenvalues and eigenfunctions of the graph Laplacians $\Delta_m$. More specifically, for
every Dirichlet $\lambda_{j}$ of $-\Delta$ on $SG$, that is a solution of $$-\Delta u = \lambda_{j} u, \quad u_{{|V_{0}}}=0,$$ there exists an integer $m\geq 1$,
called the {\em generation of birth}, such that if $u$ is a $\lambda_j$-eigenfunction and $k\geq m$
then $u_{|V_{k}}$ is an eigenfunction of $\Delta_{k}$ with eigenvalue $\lambda_{j}^{(k)}$.  The only
possible initial values $\lambda_{j}^{(m)}$ are $2$, $5$ and $6$, and subsequent values can be obtained
from
\begin{equation}\label{spectdecrelnforevals}
    \lambda_{j}^{(k+1)} = \tfrac{ 5+ \epsilon_{k} \sqrt{25-4\lambda_{j}^{(k)}}}{2}\text{ for }k\geq m
    \end{equation}
where $\epsilon_{k}$ can take the values $\pm1$.  The sequence $\lambda_{j}^{(k)}$ is related to $\lambda_{j}$
by
\begin{equation}\label{spectdecrelnofgammaktolambda}
    \lambda_{j} = \tfrac{3}{2}\lim_{k\to \infty} 5^{k}\lambda_{j}^{(k)}.
    \end{equation}

\begin{thm}\label{nmgreen}
Let $\{\lambda_{j}\}_{j=0}^{\infty}$ be the eigenvalues of the Laplacian on $SG$. Then,
$$\nm{G}_{L^{2}}^{2}=\iint_{SG\times SG} G(x, y)^{2} dxdy=\sum_{j=0}^{\infty}\tfrac{1}{\lambda_{j}^{2}}= \tfrac{45389}{3564000}\simeq .0127$$ and $$  \int_{SG} G(x, x) dx=\sum_{j=0}^{\infty}\tfrac{1}{\lambda_{j}}= 1/6$$

\end{thm}

\begin{proof}
It is easily seen that  $G$ can be expressed as $$G(x,y)=\sum_{k=0}^{\infty}\lambda_{k}^{-1}u_{k}(x)u_{k}(y),$$ where $\lambda_{k}$ are the Laplacian eigenvalues corresponding to the eigenfunction $u_k$ \cite{str06}. 

It then follows that  $$\int_{SG} G(x, x) dx=\sum_{j}\tfrac{1}{\lambda_{j}}=\lim_{m \to \infty}\tfrac{2}{3}\sum_{j \in I_{m}}\tfrac{1}{5^{m}\lambda_{j}^{(m)}},$$ where $I_m$ is a finite set that depends on the multiplicity of the graph-Laplacian eigenvalues $\lambda_{j}^{(m)}$. We shall now find explicit formula for the above sum. 

Suppose that we know the graph-Laplacian eigenvalues  $\{\lambda_{j}^{(m)}\}$ whose multiplicity are $\{d_{j}^{(m)}\}$: if $\lambda_{j}^{(m)}=6$ it has multiplicity $\tfrac{3^{m}-3}{2}$, if $\lambda_{j}^{(m)}=5$ it has multiplicity $\tfrac{3^{m-1}+3}{2}$ and if $\lambda_{j}^{(m)}=3$ it has  multiplicity $\tfrac{3^{m-1}-3}{2}$.
Then, at the next level ($m+1$), the graph-Laplacian eigenvalues $\{\lambda_{j}^{(m+1)}\}$  come from the spectral decimation method. Moreover, for every $\lambda_{j}^{(m)}\neq 6$ with multiplicity $d_{j}^{(m)}$ there will be two eigenvalues $\lambda_{j, \pm}^{(m+1)}$ with multiplicity $d_{j}^{(m)}$ and given by $$\lambda_{j, \pm}^{(m+1)}=\tfrac{5\pm \sqrt{25 - 4\lambda_{j}^{(m)}}}{2}.$$ Note that 
$$\tfrac{1}{5^{m+1}\lambda_{j, +}^{(m+1)}}+ \tfrac{1}{5^{m+1}\lambda_{j, -}^{(m+1)}}=\tfrac{1}{5^{m}}\tfrac{1}{5}\tfrac{\lambda_{j, +}^{(m+1)} + \lambda_{j, -}^{(m+1)}}{\lambda_{j, +}^{(m+1)}\lambda_{j, -}^{(m+1)}}=\tfrac{1}{5^{m}\lambda_{j}^{(m)}}.$$ Now  $\sum_{j}\tfrac{1}{\lambda_{j}}=\lim_{m \to \infty}A_{m}$ where $A_{m}=\sum_{\lambda_{j}\neq 6}\tfrac{1}{5^{m}\lambda_{j}^{(m)}}$, since the eigenvalue $\lambda_{j}^{(m)}=6$ with multiplicity $\tfrac{3^{m}-3}{2}$ contributes $O(\tfrac{3^{m}}{5^{m}})$ to the sum. Therefore, by the spectral decimation, we have
\begin{equation}\label{Am}
A_{m+1}=\tfrac{2}{3}\tfrac{3^{m}+3}{2}\tfrac{1}{5^{m+1} 5}+ \tfrac{2}{3}\tfrac{3^{m-1}-3}{2}\tfrac{1}{5^{m+1} 3} +A_m,
\end{equation}  
with $A_1= \tfrac{2}{5}\tfrac{1}{5}(1/2 + 1/5 + 1/5) =3/25.$ Consequently,
$$\sum_{j}\tfrac{1}{\lambda_{j}}=A_1+ \sum_{m=1}^{\infty}\tfrac{3^{m-1}+1}{5^{m+2}} + \tfrac{1}{3}\sum_{m=1}^{\infty}\tfrac{3^{m-1}-1}{5^{m+1}}=1/6.$$

Similarly, if we let $B_m=\tfrac{4}{9}\sum_{\lambda\neq 6}(\tfrac{1}{5^{m}\lambda_{j}^{(m)}})^{2},$, then $$\sum_{j}^{\infty} \tfrac{1}{\lambda_{j}^{2}}=\lim_{m\to \infty}B_{m}.$$ 
One can check that $$(\tfrac{1}{5^{m+1}\lambda_{j, +}^{(m+1)}})^2+ (\tfrac{1}{5^{m+1}\lambda_{j, -}^{(m+1)}})^2=\tfrac{1}{25}\tfrac{1}{5^{2m}}\tfrac{(\lambda_{j, +}^{(m+1)})^{2} + (\lambda_{j, -}^{(m+1)})^{2}}{(\lambda_{j, +}^{(m+1)}\lambda_{j, -}^{(m+1)})^{2}}=\tfrac{1}{25}\tfrac{25-2\lambda_{j}^{(m)}}{(5^{m}\lambda_{j}^{(m)})^{2}}=(\tfrac{1}{5^{m}\lambda_{j}^{(m)}})^{2}-\tfrac{2}{25}\cdot \tfrac{1}{5^{2m}\lambda_{j}^{(m)}}.$$ Next, we can use the spectral decimation method to prove that 
\begin{equation}\label{Bm}
B_{m+1}=\tfrac{4}{9}\tfrac{3^{m}+3}{2}\tfrac{1}{(5^{m+1}\cdot 5)^{2}} + \tfrac{4}{9}\tfrac{3^{m}-3}{2}\tfrac{1}{(5^{m+1}\cdot 3)^{2}} + B_m - \tfrac{2}{25}C_m.
\end{equation}
where $C_m= \sum_{\lambda_{j}\neq 6}\tfrac{1}{5^{2m}\lambda_{j}^{(m)}}=\tfrac{3}{2}\tfrac{1}{5^{m}}A_{m}.$ Moreover,

\begin{equation}\label{Cm}
C_{m+1}=\tfrac{3^{m}+3}{2}\tfrac{1}{5^{2m+2}\cdot 5} + \tfrac{3^{m}-3}{2}\tfrac{1}{5^{2m+2}\cdot 3} + \tfrac{1}{5}C_m,
\end{equation}

with $C_1=\tfrac{1}{25}(1/2+1/5+1/5)=9/250,$ and $B_1=\tfrac{4}{9}\tfrac{1}{25}(1/4+1/25+1/25)=\tfrac{11}{3\cdot (25)^2}.$

Consequently,
$$\sum_{j}\tfrac{1}{\lambda_{j}^{2}}=B_1 + \sum_{m=1}^{\infty}\tfrac{2}{3}\tfrac{3^{m-1}+3}{5^{3}\cdot 5^{2m}}+\sum_{m=1}^{\infty}\tfrac{2}{3}\tfrac{3^{m-1}-3}{15^{2}\cdot 5^{2m}}-\tfrac{2}{25}\sum_{m=1}^{\infty}C_m=\tfrac{1}{25}\big(\tfrac{3532}{25^{2}\cdot 81\cdot 11}-\tfrac{167}{1760}\big)=\tfrac{45389}{3564000}\simeq .0127.$$
\end{proof}

\subsection{Bases of polynomials on $SG$}

For any integer $j\geq 0$, the set of polynomials of degree less than or equal to $j$ will be denoted  $\mathcal{H}^j$ and consists of the  solutions of $\Delta^{j+1} u =0$.  $\mathcal{H}^j$ is a space of dimension $3j+3$ and it has a basis $\{f_{ki}, 0\leq k\leq j; i=0,1,2\}$ characterized by
\begin{equation}
\Delta^\ell f_{ki}(q_{i'}) = \delta_{\ell,k}\delta_{i,i'},
\end{equation}
where $i'=0,1,2$ and $0\leq \ell \leq j.$
In particular, $\mathcal{H}^0$ is the space of \emph{harmonic functions}, i.e., a function $h:SG \to \R$ belongs to $\mathcal{H}^{0}$ if and only if it satisfies $\Delta h = 0$.  We refer to \cite{strus, nsty1, str06} for details on polynomials on $SG$. 

Our construction of OP on $SG$ will be based on another basis for the space  $\mathcal{H}^j$ introduced in \cite{nsty1}. Functions in this basis are called {\it monomials} and  are essentially the fractal analogs of $\tfrac{x^{j}}{j!}$.

\begin{deft}
Fix a boundary point $q_n$ for $n=0,1,2$,  The monomials $P_{ji}^{(n)}$ for $i=1,2,3$ and $j\ge0$ are defined to be the functions in $\mathcal{H}^j$ satisfying

\begin{equation}\label{defmono}
\left\{ \begin{array} {r@{\quad = \quad}l}
\Delta^m P_{ji}^{(n)}(q_n) & \delta_{m,j}\delta_{i,1}\\
\partial_n \Delta^m P_{ji}^{(n)}(q_n) & \delta_{m,j}\delta_{i,2}\\
\partial_T \Delta^m P_{ji}^{(n)}(q_n) & \delta_{m,j}\delta_{i,3},\end{array}\right.
\end{equation} with $0\leq m \leq j.$
When $n=0$ we will sometimes delete the upper index and just write $P_{ji}$ for $P_{ji}^{(0)}$.
\end{deft}

Notice that for a fixed $n=0,1, 2$, the set of monomials $\{P_{ki}, i=1, 2, 3\}_{k=0}^{j}$ forms a basis for $\mathcal{H}^{j}$. Moreover, the monomials in this basis satisfy $\Delta P_{ji}^{(n)} = P_{(j-1)i}^{(n)}$, and possess some symmetry properties. When $i=1$ or $2$, $P_{ji}^{(n)}$ is symmetric with respect to the line passing through $q_n$ and the midpoint of the side opposing $q_n$. In fact, these symmetric polynomials can be viewed as analogs of the even polynomials $t^{2k}$ on the unit interval. When $i=3$, $P_{j3}^{(n)}$ is antisymmetric with respect to the line passing through $q_n$ and the midpoint of the side opposing $q_n$.  In this case, $P_{j3}^{(n)}$ can be viewed as analogs of the odd polynomials $t^{2k+1}$.

In our construction of OP on $SG$, we will need the following result which was proved in \cite{nsty1} and which gives recursively values of the above monomials and their derivatives at boundary points.

{\bf Theorem B \cite[Theorem 2.3]{nsty1}} {\it 
For $j\ge0$, let
\begin{equation*}
\alpha_j = P_{j1}(q_1), \hspace{.1in}
\beta_j = P_{j2}(q_1), \hspace{.1in}
\gamma_j = P_{j3}(q_1), \hspace{.1in}
\eta_j = \partial_n P_{j1}(q_1)
\end{equation*}

The following recursion relations hold: 
\begin{equation}\label{recursion}
\left\{ \begin{array} {r@{\quad = \quad}c@{\quad \textrm{for}\quad}c@{\quad \geq \quad}c}
\alpha_j & \frac{4}{5^j - 5} \sum_{\ell=1}^{j-1} \alpha_{j-\ell}\alpha_{\ell} &  j&2\\ [.3ex]
\beta_j & \frac{2}{15(5^j -1)} \sum_{\ell=0}^{j-1} (3\cdot 5^{j-\ell}-5^{\ell+1}+6)\alpha_{j-\ell}\beta_{\ell}  &j& 1\\ [.3ex]
\gamma_j & 3\alpha_{j+1}&   j&1\\ [.3ex]
\eta_j & \frac{5^j +1}{2} \alpha_j +2\sum_{\ell=0}^{j-1} \eta_{\ell}\beta_{j-\ell}   & j&1, \end{array}\right.
\end{equation}
where the initial values are: $\alpha_0 = 1, \alpha_1 = 1/6, \beta_0 = -1/2, \eta_0 = 0$, and $\partial_n P_{02}(q_1) = -1/2$. 
Moreover, $$\partial_n P_{j2}(q_1)= -\alpha_j\quad {\textrm for}\quad j\geq 1\qquad {\textrm and}\quad \partial_n P_{j3}(q_1)= 3\eta_{j+1}\quad {\textrm for}\quad j\geq  0.$$}
We can use Theorem B along with the Green-Gauss formula to compute the inner product among the monomials $P^{(n)}_{kj}$.

\begin{lemma}\label{lemma1}
Consider the monomials $\{P_{ji}\}_{j\geq 0}$  for $i=1, 2, 3.$ Then the following formulas hold: 
\begin{equation}\label{piinnerprod}
\left\{ \begin{array} {r@{\quad = \quad}l}
\ip{P_{j1}}{P_{k1}} & 2 \sum_{\ell=j-m_*}^j \alpha_{j-\ell}\eta_{k+\ell+1} - \alpha_{k+\ell+1}\eta_{j-\ell}\\ [.2ex]
\ip{P_{j2}}{P_{k2}} & -2 \sum_{\ell=j-m_*}^j \beta_{j-\ell}\alpha_{k+\ell+1} - \beta_{k+\ell+1}\alpha_{j-\ell}\\ [.2ex]
\ip{P_{j3}}{P_{k3}} & 18 \sum_{\ell=j-m_*}^j \alpha_{j-\ell+1}\eta_{k+\ell+2} - \alpha_{k+\ell+2}\eta_{j-\ell+1}\\ [.2ex]
\ip{P_{j1}}{P_{k2}} & -2 \sum_{\ell=0}^j \alpha'_{j-\ell}\alpha'_{k+\ell+1} - \beta_{k+\ell+2}\eta_{j-\ell+1},\end{array}\right.
\end{equation} where $\alpha'_{k}=\alpha_k$ for $k\neq 1$ and $\alpha'_1 = -1/2$, and where $m_* = \min{(j,k)}$.

Moreover, 
\begin{equation}\label{antssinnerprod}
\ip{P_{j1}}{P_{k3}} = \ip{P_{j2}}{P_{k3}} = 0.
\end{equation}
In addition, for $i=3$ the inner products among the different monomials associated with each of the boundary points $q_{n}$, $n=0, 1, 2$ are given by
\begin{equation}\label{difinnerprodp3}
\left\{ \begin{array} {r@{\quad = \quad}l}
\ip{P_{j3}^{(n)}}{P_{k3}^{(n)}} & \ip{P_{j3}^{(0)}}{P_{k3}^{(0)}}\\[.3ex]
\ip{P_{j3}^{(n)}}{P_{k3}^{(n')}} & -\frac{1}{2}\ip{P_{j3}^{(0)}}{P_{k3}^{(0)}}  \end{array}\right.
\end{equation} for $n \neq n'$.
\end{lemma}

\begin{proof}
It is enough to prove that for each $k, j$ and any $i, i'$ we have:
\begin{equation}\label{innerprod}
\ip{P_{ji}}{P_{ki'}} =  \sum_{\ell=0}^j \sum _{n=0}^2 \left( P_{(j-\ell)i}(q_n) \partial_n P_{(k+1+\ell)i'} (q_n) - P_{(k+1+\ell)i'}(q_n) \partial_n P_{(j-\ell)i} (q_n) \right)
\end{equation}
Notice that by a symmetry argument we can assume that $k\geq j$ and proved the result by induction on $j$. 
Fix $i, i' \in \{1, 2, 3\}$ and notice that for all $k\geq 0$ Green-Gauss' formula gives:
\begin{eqnarray*}
\ip{P_{0i}}{P_{ki'}} &=& \int_{SG} P_{0i} P_{ki'} \,d\mu = \int_{SG} P_{0i}\Delta P_{(k+1)i'} \,d\mu \\[.2ex]
&=&  \int_{SG} \Delta P_{0i}P_{ki'} \,d\mu+ \sum_{n=0}^2 P_{0i}(q_n)\partial_n P_{ki'}(q_n) - P_{ki'}(q_n)\partial_n P_{0i}(q_n) \\[.2ex]
&=& \sum_{n=0}^2 P_{0i}(q_n)\partial_n P_{(k+1)i'}(q_n) - P_{(k+1)i'}(q_n)\partial_n P_{0i}(q_n),  
\end{eqnarray*}where we have used the fact that $\Delta P_{0i} =0$. This establishes~\eqref{innerprod} for $j=0$, and all $k\geq 0$ . 

Now assume that  for some $j$ and for $k\geq j$ we have 
\[\ip{P_{ji}}{P_{ki'}} =  \sum_{\ell=0}^j \sum _{n=0}^2 \left( P_{(j-\ell)i}(q_n) \partial_n P_{(k+1+\ell)i'} (q_n) - P_{(k+1+\ell)i'}(q_n) \partial_n P_{(j-\ell)i} (q_n) \right), \]
we just have to show that this last relation holds for $j+1$ and $k\geq j+1$.  But, by the Gauss-Green formula we have 
\begin{eqnarray*}
\ip{P_{(j+1)i}}{P_{ki'}} &=& \int_{SG} P_{(j+1)i}P_{ki'} \,d\mu = \int_{SG} P_{(j+1)i}\Delta P_{(k+1)i'} \,d\mu \\[.2ex]
&=& \ip{P_{ji}}{P_{(k+1)i'}} + \sum_{n=0}^2 \left( P_{(j+1)i}(q_n)\partial_n P_{(k+1)i'}(q_n) - P_{(k+1)i'}(q_n)\partial_n P_{(j+1)i}(q_n) \right) \\[.2ex]
&=& \sum_{\ell=0}^j \sum _{n=0}^2 \left( P_{(j-\ell)i}(q_n) \partial_n P_{(k+2+\ell)i'} (q_n) - P_{(k+2+\ell)i'}(q_n) \partial_n P_{(j-\ell)i} (q_n)\right)\\[.2ex]
&+& \sum_{n=0}^2 \left( P_{(j+1)i}(q_n)\partial_n P_{(k+1)i'}(q_n) - P_{(k+1)i'}(q_n)\partial_n P_{(j+1)i}(q_n) \right) \\[.2ex]
&=& \sum_{\ell=1}^{j+1} \sum _{n=0}^2 \left( P_{(j+1-\ell)i}(q_n) \partial_n P_{(k+1+\ell)i'} (q_n) - P_{(k+1+\ell)i'}(q_n) \partial_n P_{(j+1-\ell)i} (q_n) \right) \\[.2ex] 
&+& \sum_{n=0}^2 \left( P_{(j+1)i}(q_n)\partial_n P_{(k+1)i'}(q_n) - P_{(k+1)i'}(q_n)\partial_n P_{(j+1)i}(q_n) \right) \\[.2ex]
&=& \sum_{\ell=0}^{j+1} \sum _{n=0}^2 \left( P_{(j+1-\ell)i}(q_n) \partial_n P_{(k+1+\ell)i'} (q_n) - P_{(k+1+\ell)i'}(q_n) \partial_n P_{(j+1-\ell)i} (q_n) \right)
\end{eqnarray*}
This shows that~\eqref{innerprod} holds for all $j \ge0$ and $k\geq j$.

Next,~\eqref{piinnerprod} follows from~\eqref{innerprod} by symmetry consideration and by observing that the values and normal derivatives at $q_0$ vanish.  In addition, Theorem B, especially~\eqref{recursion} can be used to further simplify~\eqref{innerprod} and  establish~\eqref{piinnerprod}. 

Note that~\eqref{antssinnerprod} is trivial because of the symmetry of $P_{j1}$ and $P_{j2}$ and the anti-symmetry of $P_{j3}$.

To prove~\eqref{difinnerprodp3}, it is enough to show that  
\begin{equation*}
\ip{P_{j3}^{(0)}}{P_{k3}^{(1)}} = -\dfrac{1}{2} \ip{P_{j3}^{(0)}}{P_{k3}^{(0)}}
\end{equation*}
since the other inner products can be computed using  similar arguments.

Using Equation \eqref{innerprod}, and canceling out zero terms yields
\begin{eqnarray*}
\ip{P_{j3}^{(0)}}{P_{k3}^{(1)}} &=& 
\sum_{\ell=0}^j P_{(j-\ell)3}^{(0)}(q_2) \partial_n P_{(k+\ell+1)3}^{(1)}(q_2)  - P_{(k+\ell+1)3}^{(1)}(q_2) \partial_n P_{(j-\ell)3}^{(0)}(q_2) \\
&=&-9 \sum_{\ell=0}^j \alpha_{j-\ell+1}\eta_{k+\ell+2} - \alpha_{k+\ell+1}\eta_{k-\ell+1}\\
&=&-\dfrac{1}{2}  \ip{P_{j3}^{(0)}}{P_{k3}^{(0)}}
\end{eqnarray*}
\end{proof}

\section{Orthogonal polynomials on $SG$}\label{sec2}
We are now ready to construct the families of orthogonal polynomials on $SG$. These OP will be obtained by applying the Gram-Schmidt algorithm on the polynomials $\{P_{jk}^{(n)}\}_{j\geq 0}$, where $k=1, 2, 3$, and $n=0, 1,$ or $2$. In  Subsections~\ref{subsec3}, and ~\ref{subsec4}, we plot the graphs of  various OP obtained from $\{P_{jk}\}_{j\geq 0}$, and $k=1, 2, 3$. As mentioned in the introduction, more graphs of these OP can be found at \cite{webpage}. 

But we first derive  a three-term recursion formula for our OP and give  estimates on the size of the coefficients appearing in this recursion formula. These coefficients are extremely small. This creates major instability problems when plotting the corresponding  OP. To circumvent this problem, we carry out our numerical simulation  to arbitrary precision. A similar phenomenon was already observed in \cite{nsty1}, and rational arithmetic was used. We also considered the use of rational arithmetic, but because it involves a  prohibitive time cost, we settled for arbitrary precision instead.

\subsection{General theory of orthogonal polynomials on $SG$}

\begin{deft}\label{def2}
Fix $k=1, 2$ or $3$ and denote by $\{p_j\}_{j=0}^{\infty}:=\{p_{jk}\}_{j=0}^{\infty}$  the orthogonal polynomials obtained from $\{P_{jk}\}_{j=0}^{\infty}$ by the Gram-Schmidt process, i.e., $p_j=P_{jk}-\sum_{l=0}^{j-1}d_{l}^{2}\ip{P_{jk}}{p_l}p_l$ for each $j\geq 1$. Consequently, there exists a set of coefficients   $\{\omega_{j,l}\}_{l=0}^{j}$, with $\omega_{j,j}=1$, and such that 
\begin{equation}\label{antisymop}
p_{j}(x)= P_{jk}-\sum_{l=0}^{j-1}d_{l}^{2}\ip{P_{jk}}{p_l}p_l      =P_{jk}(x) + \sum_{l=0}^{j-1}\omega_{j, l}P_{lk}(x), j\geq 1 .
\end{equation}  Moreover,  
$$\ip{p_j}{p_\ell}=d^{-2}_{j}\delta_{j,\ell}\, \textrm{where}\, \nm{p_j}_{L^{2}}^{2}=d_{j}^{-2}.$$

By normalizing the orthogonal polynomials $\{p_j\}_{j=0}^{\infty}$ we obtain the family of orthonormal polynomials
 $\{ Q_j \}_{j=0}^{\infty}$  characterized by 
\begin{equation*}
\left\{ \begin{array} {r@{\quad = \quad}l}
\langle Q_j, Q_k \rangle  & \delta_{j,k}\\
Q_j & d_j p_{j}=d_j P_{jk}+  d_{j}\sum_{l=0}^{j-1}\omega_{j,l}P_{lk},
j\geq 1.
\end{array}\right.
\end{equation*}
\end{deft}

\begin{thm}\label{nmpk}
Given $k=1, 2$ or $3$, and for each $j\geq 0$ the following holds:
$$\nm{p_j}_{L^{2}}=d_{j}^{-1}\leq \nm{P_{j,k}}_{L^{2}}.$$ 
Moreover,  for any $0< r< \infty$, there exist constants $c_1, c_r>0$ such that for all $j\geq 0$
\begin{equation}\label{decayestnmpk}
\nm{p_j}_{L^{2}}=d_{j}^{-1}\leq c_{1}(j!)^{-\log 5/\log 2} + c_{r}r^{-j}.
\end{equation}

In particular, $$\lim_{j\to \infty}\nm{p_{j}}_{L^{2}}=\lim_{j \to \infty} d_{j}^{-1}=0.$$
\end{thm}

\begin{proof}
Recall that by the definition of $p_j$ we can write that for $j\geq 1$, the Gram-Schmidt process yields $p_j=P_{jk}-\sum_{l=0}^{j-1}d_{l}^{2}\ip{P_{jk}}{p_l}p_l$ for each $j\geq 1$. Therefore,

$$\nm{P_{jk}}_{L^{2}}^{2}=\nm{p_j}_{L^{2}}^{2} + \sum_{l=0}^{j-1}\ip{P_{jk}}{p_{l}}^{2}=d_{j}^{-2} + \sum_{l=0}^{j-1}\ip{P_{jk}}{p_{l}}^{2}\geq d_{j}^{-2}$$ which establishes the first estimate.

Now recall from \cite[Theorem 2.7]{nsty1} and \cite[Theorem 2.13]{nsty1}, that for each $r\in (0, \infty)$, there are $c>0$ and $c_r>0$ such that 
$$ 0< \alpha_{j}< c (j!)^{-\log5/\log 2}\qquad 0<|\eta_{j}|\leq c_{r}r^{-j}\quad \textrm{for \, all} \quad j \geq 0.$$
When $k=1$, we get from Lemma~\ref{lemma1} that
\begin{align*}
\nm{P_{j1}}_{L^{2}}^{2}&=2\sum_{l=0}^{j}\alpha_{j-l}\eta_{j+l+1}-\alpha_{j+l+1}\eta_{j-l}\\
& = 2\big{|}\sum_{l=0}^{j}\alpha_{l}\eta_{2j-l+1} - \sum_{l=0}^{j}\alpha_{2j-l+1}\eta_{l}\big{|}\\
&\leq 2c\, c_{r}r^{-j-1}\sum_{l=0}^{j}(l!)^{-\log 5/\log 2} + 2c\, c_{r}((j+1)!)^{-\log 5/\log2}\sum_{l=0}^{j}r^{-\log 5/\log 2}
\end{align*}
from which~\eqref{decayestnmpk} follows. When $k=2$ or $k=3$  similar arguments are used to obtain the same estimate. It then follows that $\lim_{j \to \infty}d_{j}^{-1}=0$ with at least an  exponential rate of decay. 
\end{proof}

Our first main result is the following theorem that establishes the three-term recursion formula for OP on $SG$. 

\begin{thm}\label{main1}
Let $\{p_k\}_{k=0}^{\infty}$ be the  orthogonal polynomials defined
above. Let $f_0(x)=0$ and for $k\geq 0$, let $f_{k+1}$ be the
polynomial defined by  
\begin{equation}\label{fk}
 f_{k+1}(x):=-\int G(x,y)p_{k}(y)\,d\mu(y).
 \end{equation} Set $p_{-1}(x):=0,$ and $p_{0}(x)=P_{03}(x)$. Then, for each $k\geq 0$
 \begin{equation}\label{threetermpk}
 p_{k+1}(x)=f_{k+1}(x) - b_k p_{k}(x) - c_{k} p_{k-1} (x),
\end{equation} where 

\begin{equation}\label{bkck}
\left\{ \begin{array} {r@{\quad = \quad}l}
 b_{k}&d_{k}^{2}\ip{f_{k+1}}{p_{k}},\\
 c_{k}&\tfrac{d_{k-1}^{2}}{d_{k}^{2}}=\tfrac{\nm{p_{k}}_{L^{2}}^{2}}{\nm{p_{k-1}}_{L^{2}}^{2}}.
 \end{array}\right.
 \end{equation}
 Consequently,
\begin{equation}\label{cknmpk}
d_{k}^{-2}=\nm{p_{k}}_{L^{2}}^{2}=d_{0}^{-2}c_1c_2c_3\hdots c_{k-1}c_{k}.
\end{equation} 
\end{thm}

\begin{proof}
Using the definition of $f_{k+1}$ we can write \begin{equation*}
\langle f_{k+1}, p_j \rangle = -\int \int G(x,y)p_k(y) p_j(x) d\mu(y)d\mu(x) = \langle f_{j+1}, p_k \rangle
\end{equation*}
Notice that $f_{j+1}$ is a polynomial of degree $j+1$. Thus, $f_{j+1}$
is orthogonal to all $p_k$ with $j+1 <k$. Therefore, $f_{k+1}=a_k
p_{k+1} +b_{k}p_k + c_k p_{k-1}$ for some coefficients $a_k, b_k,$ and
$c_k$. It is easy to see that $a_k=1$ this follows from the fact
$f_{k+1} - p_{k+1}$ is orthogonal to $p_{k+1}$. Thus, $f_{k+1}=
p_{k+1} +b_{k}p_k + c_k p_{k-1}$. Taking inner product with $p_k$
yields the first equation in~\eqref{bkck}. Taking inner product with $p_{k-1}$ yields

\begin{align*}
c_{k}\ip{p_{k-1}}{p_{k-1}}&=c_{k}d_{k-1}^{-2}=\ip{f_{k+1}}{p_{k-1}}\\
&=-\int_{SG}\int_{SG}G(x,y)p_{k}(y)p_{k-1}(x)\, d\mu(x)\, d\mu(y)\\
&=\ip{f_{k}}{p_{k}}=\ip{p_{k}+b_{k-1}p_{k-1}+c_{k-1}p_{k-2}}{p_{k}}\\
&=d_{k}^{-2}
\end{align*}
which is exactly the second equation in~\eqref{bkck}.

Observe that the last equation in~\eqref{bkck} is equivalent to $c_{k}d_{k}^{2} = d_{k-1}^{2}$, which implies~\eqref{cknmpk}.
\end{proof}

The above results deserve some discussions. While~\eqref{threetermpk}
resembles~\eqref{classic3tr}, these two relations are fundamentally
different. In fact,~\eqref{classic3tr} is essentially the statement
that $t\pi_{k}$ is a monomial of degree $k+1$, and thus can be
expressed in terms of $\pi_l$, $l\leq k+1$. Because the product of
polynomials is not a polynomial on $SG$, $t\pi_k$ is replaced by $f_{k+1}(x)=-\int G(x,y)p_{k}(y)\,d\mu(y)$. Therefore, unlike in the classical case in which all information on $\pi_{k+1}$ can be gathered from the coefficients in~\eqref{classic3tr}, on $SG$ we must also evaluate the auxiliary polynomial $f_{k+1}$. This presents an additional difficulty in carrying out any numerical simulations with our OP.

We now prove a version of Theorem~\ref{main1} dealing with the three-term recurrence for the orthonormal polynomial $\{Q_{k}\}_{k=0}^{\infty}$.

\begin{thm}\label{main2}
Let $\{Q_k\}_{k=0}^{\infty}$ be the orthogonal polynomials defined above. Let $\tilde{f}_{0}=0$ and for $k\geq 0$
\begin{equation}\label{tildefk}
 \tilde{f}_{k+1}(x):=-\int G(x,y)Q_{k}(y)\,d\mu(y).
 \end{equation}  Then, for each $k\geq 0$
 
\begin{equation}\label{threetermQk}
 \sqrt{c_{k+1}}Q_{k+1}(x)=\tilde{f}_{k+1}(x) - b_k Q_{k}(x) - \sqrt{c_{k}} Q_{k-1} (x),
\end{equation}
 where $Q_{-1}(x):=0,$  $Q_{0}(x)=d_{0}P_{03}(x),$ and $b_k$ and $c_k$ were defined in Theorem~\ref{main1}. 
\end{thm}

\begin{proof}
The proof follows from the fact that $p_k=\nm{p_k}_{L^{2}}Q_k$ and~\eqref{ck}. 
\end{proof}

We prove below certain properties about the coefficients $b_k$ and $c_k$ appearing in the three-term recursion formula.

\begin{thm}\label{boundbkck} For each $k\geq 0$, we have 
\begin{equation}\label{boundbk}
-\|G\|_{L^{2}}\leq b_k <0, 
\end{equation}
\begin{equation}\label{boundck}
0< c_k \leq \|G\|_{L^{2}}^2,
\end{equation}
and 
\begin{equation}\label{bounddk}
d_{k+1}^{-1} \leq \|G\|_{L^{2}}d_{k}^{-1}.
\end{equation}
In particular, 
$$d_{k}^{-1}\leq d_{0}^{-1}\nm{G}_{L^{2}}^{k}, $$ where $\nm{G}_{L^{2}}=\bigg(\iint_{SG\times SG}G(x,y)^{2}dx dy\bigg)^{1/2}.$
\end{thm}

\begin{proof} $|b_{k}|=d_{k}^{2}|\ip{f_{k+1}}{p_{k}}|\leq
  d_{k}^{2}\nm{f_{k+1}}_{L^{2}}\nm{p_{k}}_{L^{2}}\leq
  d_{k}^{2}\nm{G}_{L^{2}}\nm{p_{k}}_{L^{2}}^{2}=\nm{G}_{L^{2}}.$ 

Using the expression of $G(x, y)$ given in the proof of Theorem~\ref{nmgreen}, we can write  $$\ip{f_{k+1}}{p_{k}}=-\sum_{l=1}^{\infty}\lambda^{-1}_{l}\ip{p_{k}}{u_{l}}^{2} < 0.$$ Hence, $b_k <0$. 
Similarly, 
$$0<c_{k}d^{-2}_{k-1}=\ip{f_{k+1}}{p_{k-1}}\leq \nm{f_{k+1}}_{L^{2}}\nm{p_{k-1}}_{L^{2}}\leq \nm{G}_{L^{2}}\nm{p_{k}}_{L^{2}}\nm{p_{k-1}}_{L^{2}}=\nm{G}_{L^{2}}d_{k}^{-1}d_{k-1}^{-1}.$$ Using the fact that $c_{k}=d_{k}^{-2}d_{k-1}^{2}$, we have $$ 0< c_k \leq \nm{G}_{L^{2}}d_{k-1}d_{k}^{-1}=\nm{G}_{L^{2}}\sqrt{c_{k}},$$ from which the second estimate follows.

The last estimate follows by observing that  $$0<d_{k+1}^{-2}< \nm{f_{k+1}}_{L^{2}}^{2}=d_{k+1}^{-2}+b_{k}^{2}d_{k}^{-2}+c_{k}^{2}d_{k-1}^{-2}\leq \nm{G}_{L^{2}}^{2}d^{-2}_{k}.$$ Consequently, $$d_{k}^{-1}\leq d_{0}^{-1}\nm{G}_{L^{2}}^{k}.$$
\end{proof}

\begin{rem}
The last estimate in Theorem~\ref{boundbkck} along with the estimate on the $\|G\|_{L^{2}}$ given in Theorem~\ref{nmgreen}, gives another proof that the sequence $d_{k}^{-1}$ decays exponentially fast. 
\end{rem}

\subsection{Recurrence relations for orthogonal polynomials on $SG$}

We have now proved via two different arguments that the coefficients
$d_{k}^{-1}$ are very small. This, together with the fact that we must also evaluate the {\it auxiliary} polynomial $f_{k+1}$ prevents us to use Theorem~\ref{main1} to directly plot the OP we construct. Rather we use~\eqref{threetermpk} together with the representation of $p_k$ in Definition~\ref{def2} to plot these OP. More specifically, we identify each $p_j$ with a vector $\Omega_{j}\in \R^{j+1}$ where $\Omega_{j}(l)=\omega_{j, l}$ for $l=0, 1, \hdots, j-1$ and $\Omega_{j}(j)=1$. This can be simply written as 
\[
{p}_j = \sum_{\ell=0}^j {\omega_{j,\ell}} P_{\ell,k} \leftrightarrow \Omega_{j}=\left( \begin{array}{c} {\omega}_{j,0} \\ {\omega}_{j,1} \\ \vdots \\ {\omega}_{j,j}\end{array}\right) =\left( \begin{array}{c} {\omega}_{j,0} \\ {\omega}_{j,1} \\ \vdots \\ 1\end{array}\right) 
\]

We now derive a recursion relation for the coefficients $\Omega_j$  appearing in the above representation of $p_j$. We shall later, implement this recurrence using arbitrary precision to give accurate graphs of $Q_j=d_{j}^{-1}p_j$.

\begin{thm}\label{pkjet}
The recurrence relation for the  orthogonal polynomials $p_k$ is accomplished by evaluating the following equations for each $k\geq 0$.
$$\Omega_{k+1}=\left( \begin{array}{c}{\zeta}_{k} \\ {\Omega}_{k}\end{array}\right) - b_{k} \left( \begin{array}{c} {\Omega}_{k}\\ 0\end{array}\right)
-c_{k} \left( \begin{array}{c} {\Omega}_{k-1}\\ 0\\0\end{array}\right),$$ or more specifically, 
\begin{equation}\label{matrixrec}
\left( \begin{array}{c} {\omega}_{k+1,0} \\ {\omega}_{k+1,1} \\ \vdots \\ {\omega}_{k+1,k-1}\\ {\omega}_{k+1,k}\\ 1\end{array}\right) =  \left( \begin{array}{c} {\zeta}_{k} \\ {\omega}_{k,0} \\ \vdots \\ {\omega}_{k,k-2}\\ {\omega}_{k,k-1}\\1\end{array}\right) - {b}_{k} \left( \begin{array}{c} {\omega}_{k,0} \\ {\omega}_{k,1} \\ \vdots \\ {\omega}_{k,k-1}\\ 1\\0\end{array}\right) -{c}_{k} \left( \begin{array}{c} {\omega}_{k-1,0} \\ {\omega}_{k-1,1} \\ \vdots \\ \omega_{k-1, k-2}\\ 1\\ 0\\ 0\end{array}\right)
\end{equation}
where $b_k, c_k$ where defined in Theorem~\ref{main1}
and $ {\zeta}_k = -\frac{1}{\gamma_0} \sum_{\ell=0}^k
\omega_{k,\ell}\gamma_{\ell+1}, $ with $\gamma_\ell$ given in~\eqref{recursion}.
\end{thm}

\begin{proof}
Using ~\eqref{fk} we can write 
\begin{eqnarray*}
f_{j+1} &=& -\int G(x,y)p_{j}(x) d\mu(y) = -\int G(x,y)\sum_{\ell=0}^{j} \omega_{j,\ell}P_{\ell,k} d\mu(y) \\
 &=&  \sum_{\ell=0}^{j} \omega_{j,\ell} \bigg(-\int G(x,y)P_{\ell,k} d\mu(y)\bigg) = \sum_{\ell=0}^{j} \omega_{j,\ell} \left( P_{\ell+1,j} - \frac{\gamma_{\ell+1}}{\gamma_0} P_{0,k}   \right) \\
 &=&  \zeta_{j} P_{0,k} + \sum_{\ell=0}^{j} \omega_{j,\ell} P_{\ell+1,k}  \\  
&=& \left( \begin{array}{c} {\zeta}_{j} \\ {\omega}_{j,0} \\ \vdots \\ {\omega}_{j,j-2}\\ {\omega}_{j,j-1}\\ 1\end{array}\right)
\end{eqnarray*}
\end{proof}

We now derive a version of the Christoffel-Darboux formulas for the OP $\{Q_k\}$. As an application, we use these formulas to find the expansion of $\Delta Q_{k}$ in terms of $\{Q_{l}\}_{l=0}^{k-1}$. 

\begin{thm}\label{christoffel-darboux} Let $N\geq 0$,  $x, y \in SG$, and define $K_{N}(x, y)=\sum_{k=0}^{N}Q_{k}(x)Q_{k}(y)$. Using the notations of the last theorem we have 
\begin{equation}\label{cd1}
K_{N}(x, y)=\sqrt{c_{N+1}}[Q_{N}(x)\Delta Q_{N+1}(y) - Q_{N+1}(x)\Delta Q_{N}(y)] + \sum_{k=0}^{N}\tilde{f}_{k+1}(x)\Delta Q_{k}(y).
\end{equation}
In particular, for each $x \in SG$ and $N\geq 0$, 
\begin{equation}\label{cd2}
K_{N}(x,x)= \sqrt{c_{N+1}}[Q_{N}(x)\Delta Q_{N+1}(x) - Q_{N+1}(x)\Delta Q_{N}(x)] + \sum_{k=0}^{N}\tilde{f}_{k+1}(x)\Delta Q_{k}(x).
\end{equation}
\end{thm}

\begin{proof}
Let $k \in \{0, 1, \hdots, N\}$, and $x, y \in SG$. By~\eqref{threetermQk} we have 
\begin{align*}
\tilde{f}_{k+1}(x)Q_{k}(y) - \tilde{f}_{k+1}(y)Q_{k}(x)& = \sqrt{c_{k+1}}[Q_{k+1}(x)Q_{k}(y)-Q_{k+1}(y)Q_{k}(x)]-\\
&\qquad \qquad \sqrt{c_{k}}[Q_{k}(x)Q_{k-1}(y)-Q_{k}(y)Q_{k-1}(x)].
\end{align*} Therefore, summing both sides we obtain
\begin{equation}\label{cde1}
\sum_{k=0}^{N}[\tilde{f}_{k+1}(x)Q_{k}(y) - \tilde{f}_{k+1}(y)Q_{k}(x)]=\sqrt{c_{N+1}}[Q_{N+1}(x)Q_{N}(y)-Q_{N+1}(y)Q_{N}(x)].
\end{equation} Now, observe that 
\begin{align*}
\sum_{k=0}^{N}[\tilde{f}_{k+1}(x)Q_{k}(y) - \tilde{f}_{k+1}(y)Q_{k}(x)] & =&  \sum_{k=0}^{N}\tilde{f}_{k+1}(x)Q_{k}(y) + \int_{SG}\sum_{k=0}^{N}G(y, z)Q_{k}(z)Q_{k}(x)dz\\
& =&  \sum_{k=0}^{N}\tilde{f}_{k+1}(x)Q_{k}(y) + \int_{SG}G(y, z)K_{N}(z, x)dz
\end{align*}
Using this last equation and taking the Laplacian with respect to $y$ of both sides of~\eqref{cde1} yields,
\begin{align*}
\Delta_{y}\bigg( \sum_{k=0}^{N}[\tilde{f}_{k+1}(x)Q_{k}(y) - \tilde{f}_{k+1}(y)Q_{k}(x)]\bigg) &= \sqrt{c_{N+1}}[Q_{N+1}(x)\Delta Q_{N}(y)- Q_{N}(x)\Delta Q_{N+1}(y))\\
& = \sum_{k=0}^{N}\tilde{f}_{k+1}(x)\Delta Q_{k}(y)- K_{N}(x,y)
\end{align*} Consequently,
\begin{equation*}
K_{N}(x,y)= \sqrt{c_{N+1}}[Q_{N}(x)\Delta Q_{N+1}(y) - Q_{N+1}(x)\Delta Q_{N}(y)] + \sum_{k=0}^{N}\tilde{f}_{k+1}(x)\Delta Q_{k}(y).
\end{equation*}
\end{proof}

We now use this theorem to find the coordinates of $\Delta Q_k$ in terms of lower order polynomials. That is since, $\Delta Q_k \in span\{Q_{l}\}_{l=0}^{k-1}$, we have $$\Delta Q_{k}(x)=\sum_{l=0}^{k-1}A_{l}^{(k)}Q_{l}(x),$$  where $A_{l}^{(k)} =\ip{\Delta Q_k}{Q_l}$. These coefficients can be computed recursively using Theorem~\ref{christoffel-darboux}. More specifically, we have

\begin{cor}\label{cd-cor} Let  $k\geq 1$, and assume $\Delta Q_{k}(x)=\sum_{\ell=0}^{k-1}A^{(k)}_{\ell}Q_{\ell}(x).$ Then,  
$A^{(k)}_{k-1}=\tfrac{1}{\sqrt{c_{k}}},$ and for $\ell=0, 1, \cdots, k-2$ we have $$A^{(k)}_{\ell}=-A^{(k-1)}_{\ell}\tfrac{b_{k-1}}{\sqrt{c_{k}}},$$ with the initial condition $\Delta Q_{1}=A^{(1)}_{0}Q_0$ where  $A^{(1)}_{0}=\tfrac{1}{\sqrt{c_{1}}}.$

Furthermore, for each $k\geq 1$ and $\ell=0, 1, \hdots, k-1$, $A_{\ell}^{(k)}>0.$
\end{cor}

\begin{proof}
Use $N=0$ in~\eqref{cd2}, gives $K_{0}(x,x)=Q_{0}^{2}(x)=\sqrt{c_{1}}Q_{0}(x)\Delta Q_{1}(x)$, and integrating this equality leads to $1=\sqrt{c_{1}}\ip{\Delta Q_{1}}{Q_{0}}$. Hence, $\Delta Q_{1}(x)=A^{(1)}_{0}Q_{0}(x)=\tfrac{1}{\sqrt{c_{1}}}Q_{0}(x)$. 

Next, write $\Delta Q_{2}(x)=A^{(2)}_{1}Q_{1}(x) + A^{(2)}_{0}Q_{0}(x)$. Using  $N=1$ in~\eqref{cd2} gives $$K_{1}(x,x)=\sqrt{c_{2}}[Q_{1}(x)\Delta Q_{2}(x) -Q_{2}(x)\Delta Q_{1}(x)] +\tilde{f}_{2}(x)\Delta Q_{1}(x)$$ where we use the fact that $\Delta Q_0=0$. Integrating this last expression leads to
$$2=\sqrt{c_{2}}[\ip{\Delta Q_{2}}{Q_{1}}-\tfrac{1}{\sqrt{c_{1}}}\ip{Q_{2}}{Q_{0}}]+\tfrac{1}{\sqrt{c_{1}}}\ip{\tilde{f}_{2}}{Q_{0}}=\sqrt{c_{2}} A^{(2)}_{1} +\tfrac{1}{\sqrt{c_{1}}} \sqrt{c_{1}}.$$ Therefore, $A^{(2)}_{1}=\tfrac{1}{\sqrt{c_{2}}}$. 

To compute $A^{(2)}_{0}$ we utilize~\eqref{cd1} with $N=1$. In particular, we multiply $K_{1}(x,y)$ by $Q_{1}(x)Q_{0}(y)$ and integrate the resulting equation with respect to both $x$ and $y$. This will give
$$0=\sqrt{c_{2}}\ip{\Delta Q_{2}}{Q_{0}} +\tfrac{1}{\sqrt{c_{1}}}\ip{\tilde{f}_{2}}{Q_1}\ip{Q_0}{Q_0}=\sqrt{c_2}A^{(2)}_{0} +\tfrac{b_1}{\sqrt{c_1}}.$$
Consequently, $A^{(2)}_{0}=-\tfrac{1}{\sqrt{c_2}}\tfrac{b_1}{\sqrt{c_1}}=-A^{(1)}_{0}\tfrac{b_1}{\sqrt{c_1}}.$ Observe that $\tilde{f}_{2}(x)=\sqrt{c_{2}} Q_{2}(x)+ b_{1} Q_{1}(x) +\sqrt{c_1}Q_{0}(x)$ was used in the above arguments.

Assume now that we have determined the coefficients $\{A^{(k)}_{\ell}\}_{\ell=0}^{k-1}$ of $\Delta Q_{k}$ with respect to the orthonormal set $\{Q_{\ell}\}_{\ell=0}^{k-1}$. We can use an induction argument to find the coordinates of $\Delta Q_{k+1}$ in the orthonormal set $\{Q_{\ell}\}_{\ell=0}^{k}$. Set $N=k$ in\eqref{cd2}, and integrate the resulting equation to obtain $A^{(k+1)}_{k}=\tfrac{1}{\sqrt{c_{k+1}}}$.

 To obtain the remaining coefficients, $A^{(k+1)}_{l}$ for $l=0, 1, \cdots, k-1$, set $N=k$ in~\eqref{cd1}, multiply the resulting  equation by $Q_{k}(x)Q_{l}(y)$ and integrate with respect to both $x$ and $y$.  
 
 To obtain the positivity of the coefficients $A_{\ell}^{(k)}$, we proceed by induction on $k$. Clearly, $A_{0}^{(1)}=\tfrac{1}{\sqrt{c_1}}>0$, and $A_{0}^{(2)}=-A_{0}^{(1)}\tfrac{b_1}{\sqrt{c_2}}>0$ since $b_k<0$ for all $k$. In addition, $A_{1}^{(2)}=\tfrac{1}{\sqrt{c_2}}>0$. For assume, that for $k\geq 2$, and each $\ell=0, 1, \hdots, k-1$, we have $A_{\ell}^{(k)}>0$. Then, $A_{k}^{(k+1)}=\tfrac{1}{\sqrt{c_{k+1}}}>0$ and for each $\ell=0, 1, \hdots, k-1$ we have $A_{\ell}^{(k+1)}=-A_{\ell}^{(k)}\tfrac{b_{k}}{\sqrt{c_{k+1}}}>0$ since $b_k <0$ and $A_{\ell}^{(k)}>0$ by the induction hypothesis. This concludes the proof. 
\end{proof}

\begin{rem} Note that Theorem~\ref{christoffel-darboux} and Corollary~\ref{cd-cor} hold for the symmetric OP $S_k$ defined in Subsection~\ref{subsec4}. 
\end{rem}

We end this section with an investigation of the tri-diagonal symmetric Jacobi matrix corresponding to the OP $\{Q_{k}\}$. 
As in the classical case of OP defined on the real line, each of the families of orthogonal polynomials we constructed above is related to  a tri-diagonal, symmetric Jacobi matrix that we denote $J_n$.  In particular, we can write~\eqref{threetermQk} of Theorem~\ref{main2}  in a matrix form that will involve this Jacobi polynomial. To do this we use the following notations:
$$\tilde{F}_{n}=[\tilde{f}_{0}, \tilde{f}_{1}, \hdots, \tilde{f}_{n-1}]^{T}\qquad Q^{(n)}=[Q_0, Q_1, \hdots, Q_{n-1}]^{T},$$ and 
\[J_n=\left( \begin{array}{ccccccc} b_0& \sqrt{c_1}& 0 & 0& \cdots & 0&0\\ \sqrt{c_1}& b_1 & \sqrt{c_2}& 0 & \cdots & 0 & 0\\ 0& \sqrt{c_2}& b_2& \sqrt{c_3} &\cdots & 0&0\\ \multicolumn{7}{c}{\dotfill}\\ 0& 0 & 0& 0& \cdots &\sqrt{c_{n-1}}& b_{n-1}\end{array} \right) \]
By writing~\eqref{threetermQk} in a matrix-vector form we have for each $n\geq 2$
\begin{equation}\label{jacobieq}
\tilde{F}_{n}(x)=J_{n} Q^{(n)}(x) + \sqrt{c_{n}}Q_{n}(x)e_{n}
\end{equation}
where $e_n=[0, 0, \hdots, 0, 1]^{T}$.

For the OP we consider in this paper, though $J_n$ is invertible for all $n$, its determinant is extremely small. Recall \cite[Theorem 0.9.10]{hojo} that the determinant of a tri-diagonal symmetric Jacobi matrix obeys the following recursion formula. 
Let $D_n=Det(J_n)$, then 
\begin{equation*}
D_{n+1} = b_{n-1} D_n - c_{n-1} D_{n-1}
\end{equation*}
with initial conditions  $D_0 = b_0$ and $D_{-1} = 1$.

 To illustrate the fact that the Jacobi matrices associated to the OP considered here have small determinant, we show below a graph of $\log |D_{n}|$ as a function of $n$. This graph is generated using the antisymmetric OP that will be constructed explicitly in the next section. 

\begin{figure}[h]
\begin{center}
\includegraphics[scale=.5]{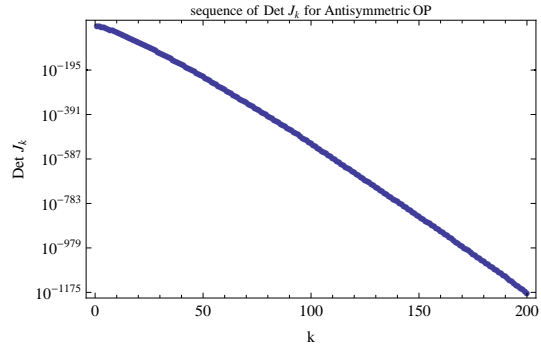}
\caption{Graph of $Log|D_k|$ of the determinants of the Jacobi matrix, $D_k$ }
\label{fig:jacobidet}
\end{center}
\end{figure}

\section{Numerical results}\label{sec3}
We shall now generate graphs of the OP we constructed above starting from two families of monomials on SG. We also present graphs of the sequences appearing in the three-terms recursion formula.  We start by considering a family of antisymmetric monomials. 
\subsection{Antisymmetric orthonormal polynomials}\label{subsec3}

Recall the monomials  $\{P_{j3}\}$ in Definition~\ref{def2} are centered around the point $q_0$ and  antisymmetric with respect to that point.  
We now use these monomials to generate the family of antisymmetric orthogonal polynomials $ \{p_{j}\}_{j=0}^{\infty}$ and their corresponding normalized version $\{Q_{j}\}_{j=0}^{\infty}$. We do this inductively using Theorems~\ref{main1} and ~\ref{pkjet}. In the process, we compute the sequence $d_k, c_k$ and $b_k$. 

More specifically, it is easily seen that $\nm{P_{03}}_{L^{2}}^{2}=d_0^{-2}=3/10$.  Next we use Theorem~\ref{pkjet} to find the auxiliary polynomial $$f_1(x)=-\int_{SG}G(x, y) p_{0}(y)dy = -\int_{SG}G(x, y) P_{0,3}(y)dy .$$  Notice that we impose the condition $f_{j}{|_{V_{0}}}=0$. Thus,  $b_0$ can now be computed, which , together with the polynomial $f_1$, is used to find the polynomial $p_1$. Next, $f_2$ is computed again using Theorem~\ref{pkjet}, from which one gets $d_1$. The next step is to compute $c_1$. Using this, $b_1$ and the polynomial $p_2$ are calculated yielding $c_2$. Continuing these recursions, we construct $p_k$, $k\geq 2$ and all the related sequences. 

To better display the results of these computations, we use a logarithmic scale to plot the numerical sequences obtained above. In particular, we plot  $\log (\nm{p_{k}}_{L^{2}}^{-2})=\log (d_{k}^{2})$, $\log (-b_k)$, and  $\log c_{k}$ versus $k$. 
These graphs are displayed in Figure~\ref{fig:dk}. In particular, these graphs illustrate Theorems~\ref{nmpk} and~\ref{boundbkck}. That is $\nm{p_{k}}_{L^{2}}^{-2}$ tends to infinity exponentially fast. In addition the graphs of $\log (-b_k)$ and $\log c_k$ suggest that $$c_k < -b_k\, \textrm{for}\,  k>k_0$$ for some $k_0$. This seems to be the case for $k<200$, but we do not have data beyond this range of $k$ to confirm or infirm the observation.

\begin{figure}[htp]
\begin{center}
\includegraphics[scale=.5]{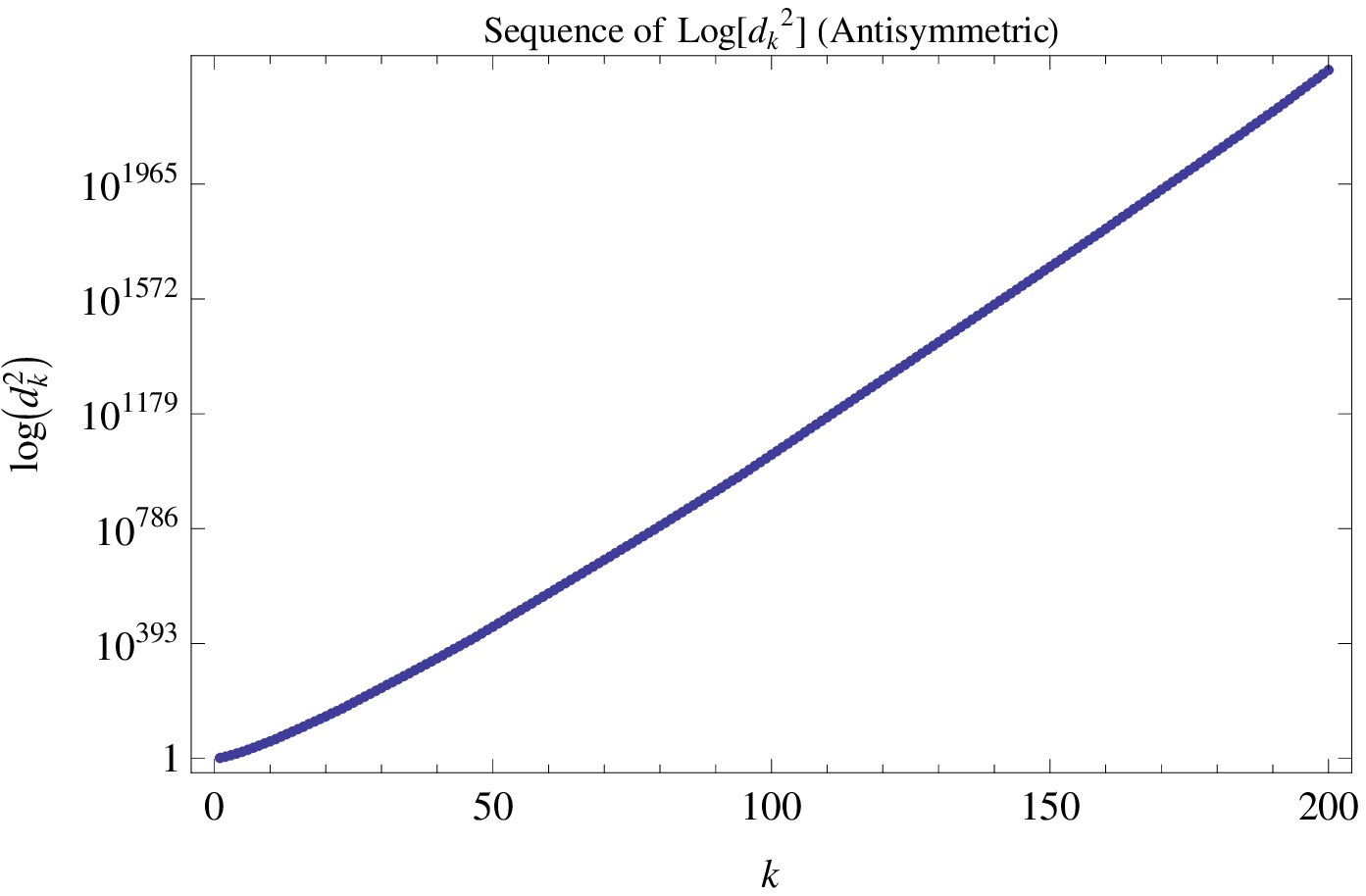} 
\includegraphics[scale=.5]{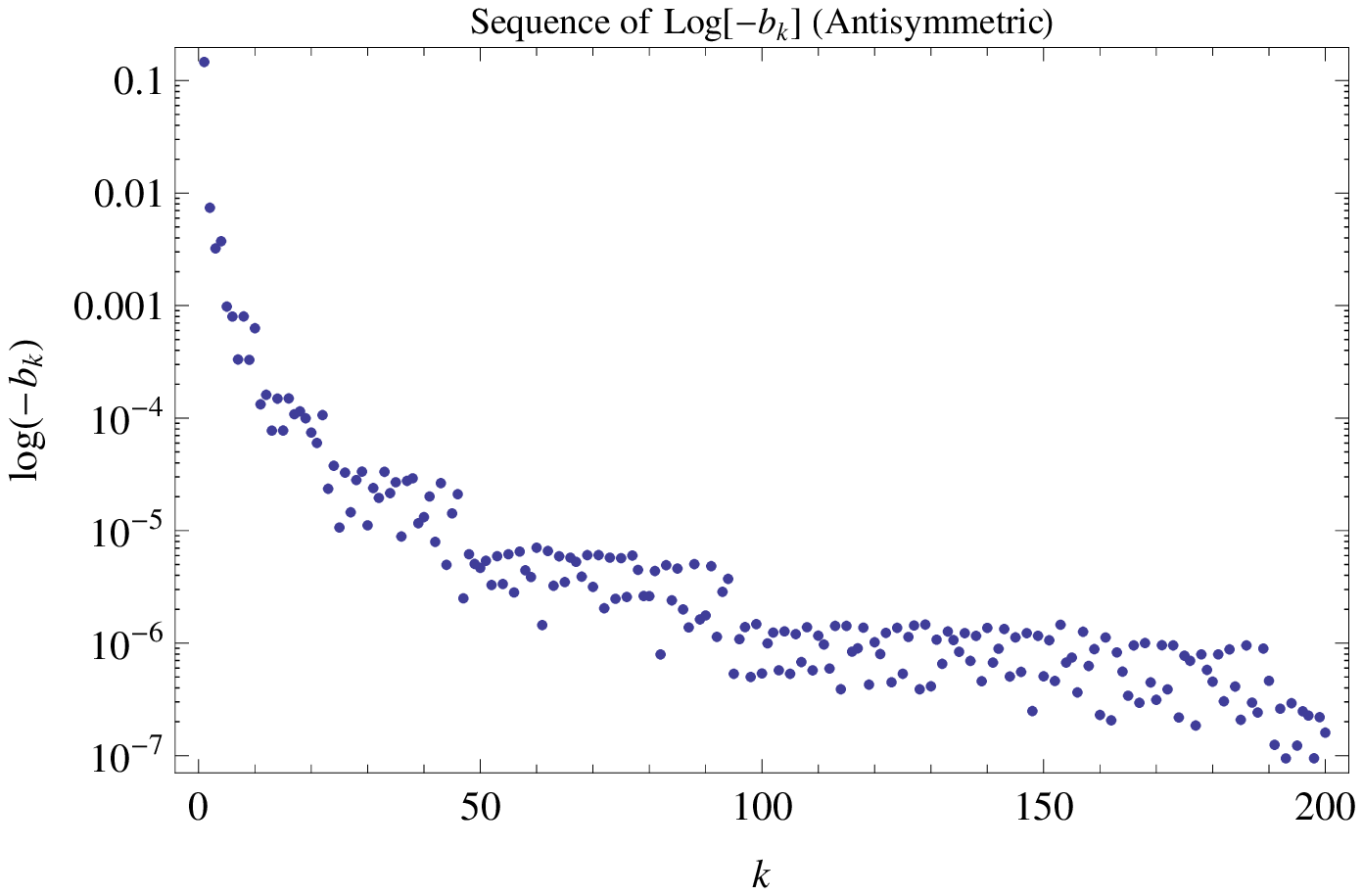}
\includegraphics[scale=.5]{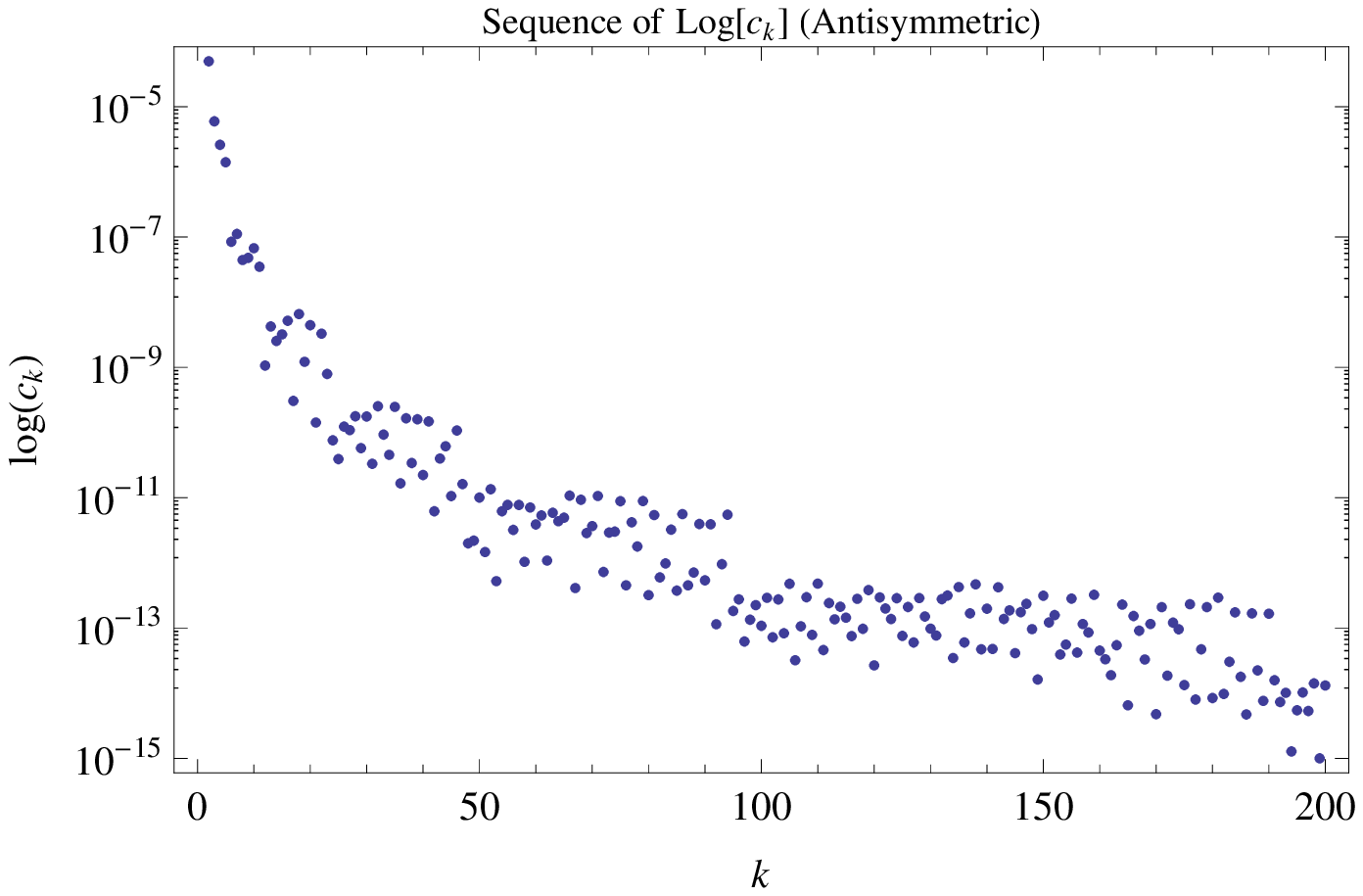}
\caption{Plots of the $\log d_k^2$, $\log (-b_k)$, and $\log c_k$ in log scale for the antisymmetric OP $p_k$.}
\label{fig:dk}
\end{center}
\end{figure}

In Figure~\ref{fig:6aop}, we plot the orthonormal polynomials $Q_j=d_{j}p_{j}$ for $j=0, 3, 4, 7$, on $\Gamma_7$ the level $7$ approximation to $SG$. To do this, we use the recursion formula given by Theorem~\ref{pkjet} along with the second half of the representation of $p_j$ given in~\eqref{antisymop}, as well as the  graphs of the monomials $P_{j,3}$ which were given in~\cite{nsty1}. Moreover,  most of the computations have been carried out  to arbitrary precision since the coefficients $\omega_{j,l}$ are very small. We refer to~Table \ref{tab:pk} for approximate values of these coefficients for $p_j$, $j=0, 1, \hdots, 6$. 

The fact that $Q_k$ are antisymmetric  is apparent from Figure~\ref{fig:6aop}. In addition, we observe that for $i=1, 2$,  $|Q_k (q_{i})|$ is larger compare to $|Q_k(x)|$ for any $x \in SG\setminus \{q_1, q_2\}$. This is consistent with the behavior of classical OP such as the Legendre polynomial on $[0, 1]$; see \cite{hos10}.

\begin{table}[htbp]
  \centering
  \caption{$p_k$ and their coefficients $\omega_{j\ell}$}
    \begin{tabular}{rrrrrrrr}
          & $P_{0,3}$    & $P_{1,3}$    & $P_{2,3}$    & $P_{3,3}$    & $P_{4,3}$    & $P_{5,3}$    & $P_{6,3}$ \\
    $p_0$    & 1     & 0     & 0     & 0     & 0     & 0     & 0 \\
    $p_1$    & -2.04E-02 & 1     & 0     & 0     & 0     & 0     & 0 \\
    $p_2$    & 1.08E-04 & -1.30E-02 & 1     & 0     & 0     & 0     & 0 \\
    $p_3$    & -3.23E-07 & 6.03E-05 & -9.74E-03 & 1     & 0     & 0     & 0 \\
    $p_4$    & 1.99E-10 & -6.41E-08 & 2.14E-05 & -6.01E-03 & 1     & 0     & 0 \\
    $p_5$    & -8.20E-14 & 5.16E-11 & -2.94E-08 & 1.41E-05 & -5.03E-03 & 1     & 0 \\
    $p_6$    & 4.43E-17 & -3.53E-14 & 2.63E-11 & -1.77E-08 & 9.99E-06 & -4.23E-03 & 1 \\    \end{tabular}
  \label{tab:pk}
\end{table}

\begin{figure}[hbp]
\begin{center}
\includegraphics[scale=.33]{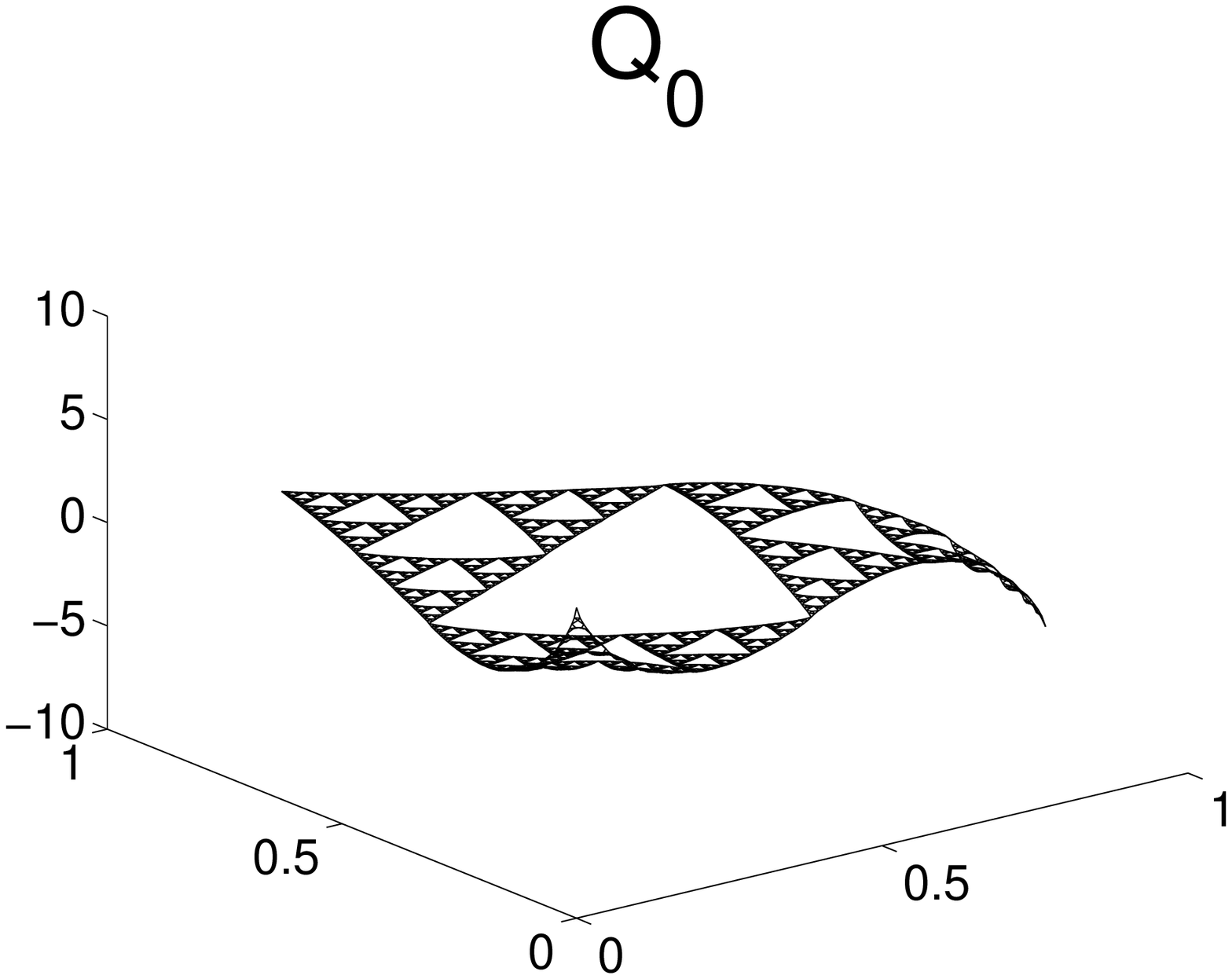} 
\includegraphics[scale=.33]{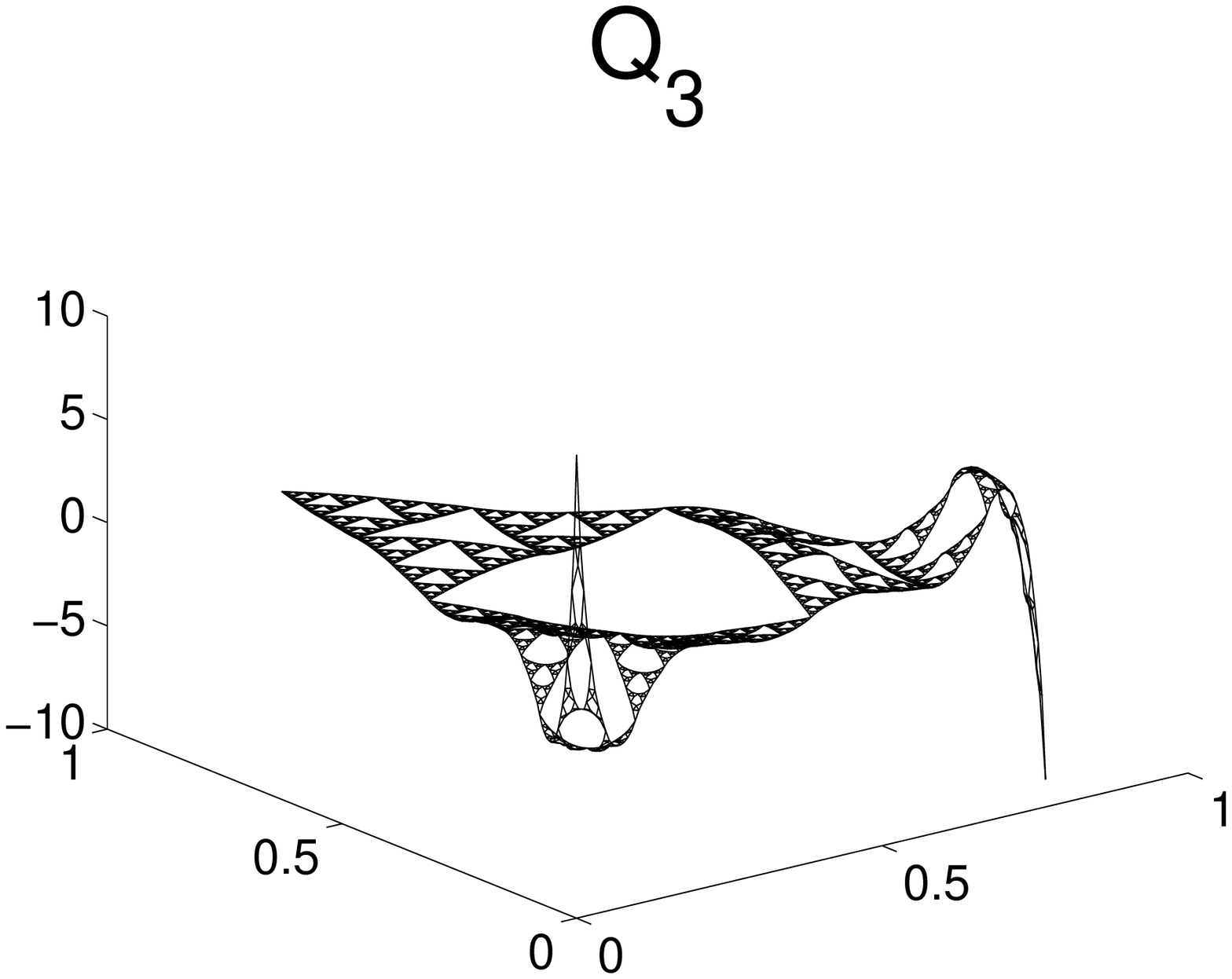}
\includegraphics[scale=.33]{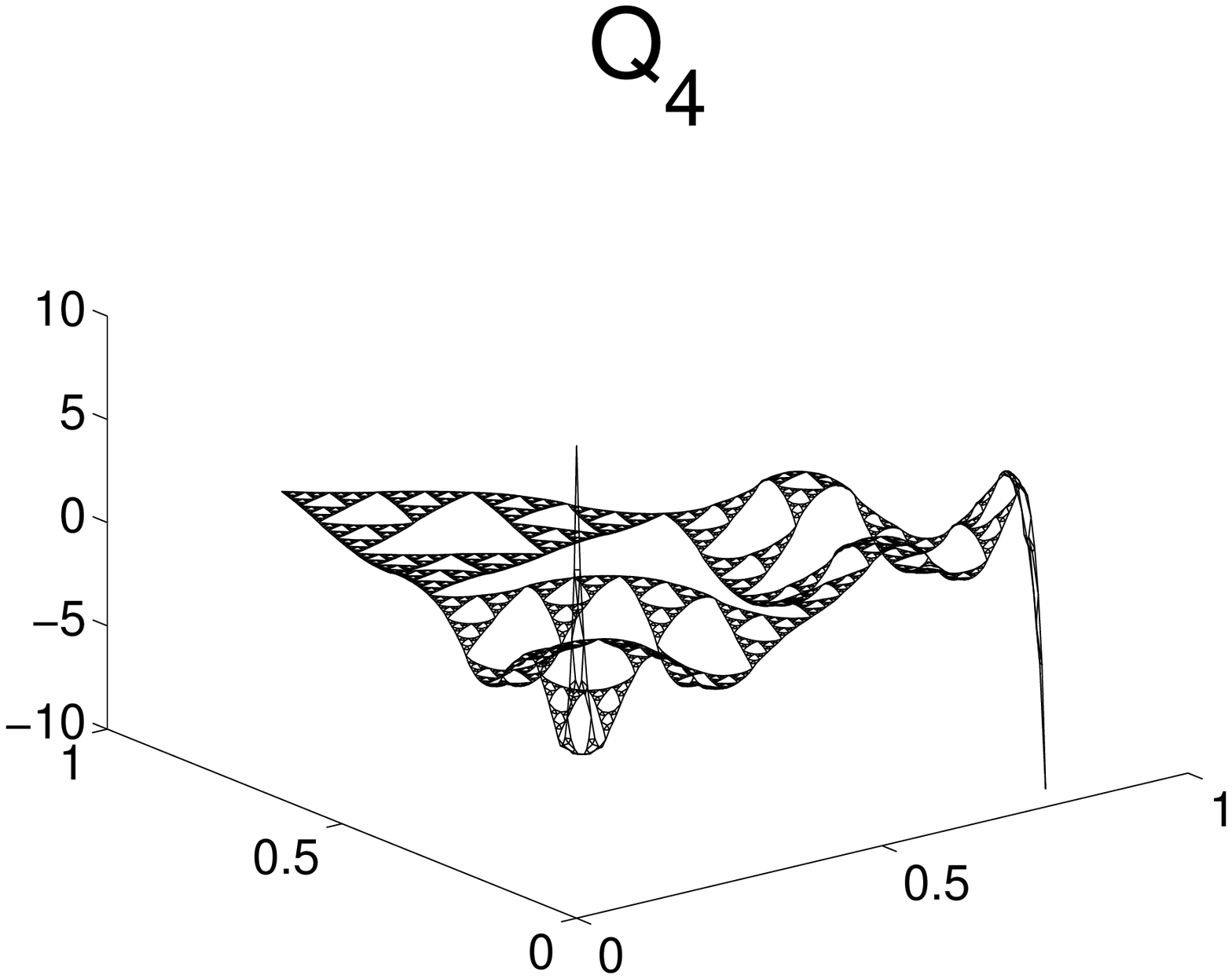} 
\includegraphics[scale=.33]{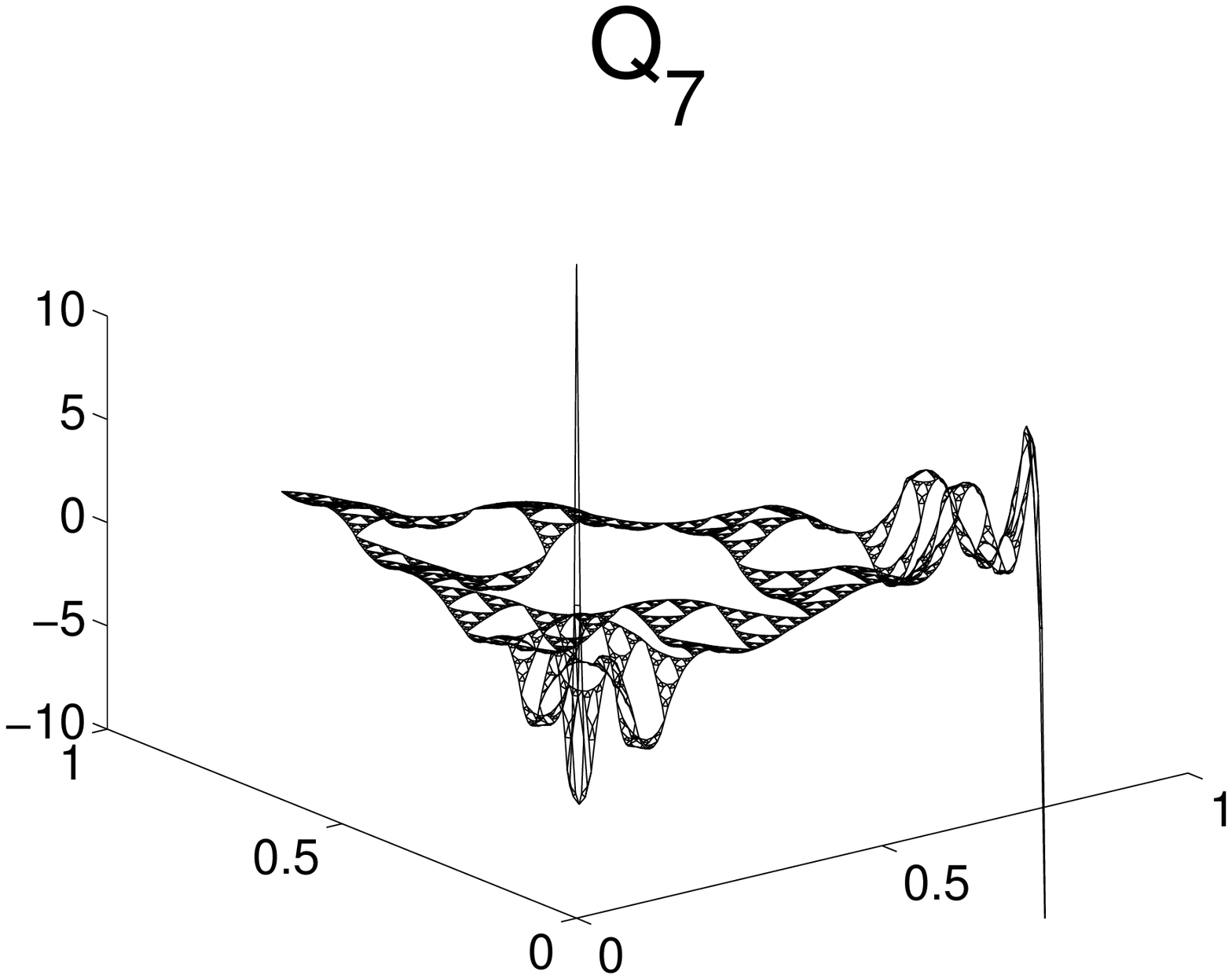}
\caption{4 antisymmetric orthonormal polynomials.}
\label{fig:6aop}
\end{center}
\end{figure}

\subsection{Symmetric orthogonal polynomials}\label{subsec4}
We next carried out the constructions of subsection~\ref{subsec3} to the fully symmetric polynomials which we first define.  By fully symmetric we mean symmetric under all rotations and flips in $D_3$.  Then we will apply the results of section~\ref{sec2} to this family of fully symmetric polynomials to obtain the symmetric OP $\{s_{k}\}_{k=0}^{\infty}$ and their normalized counterpart $\{S_{k}\}_{k=0}^{\infty}$. 

\begin{deft}\label{fspoly}
Define the fully symmetric monomial as follows for each $j \ge 0$:
\begin{equation}\label{rho}
\rho_j = P_{j1}^{(0)}+P_{j1}^{(1)}+P_{j1}^{(2)}
\end{equation}
The symmetric OP denoted $\{s_{j}\}_{j=0}^{\infty}$ are now obtained from these symmetric polynomials by applying the Gram-Schmidt process and satisfies $$\ip{s_{j}}{s_{k}}=d_{j}^{-2}\delta_{jk},$$ where $d_{j}^{-2}=\nm{s_{j}}^{2}$. Moreover, $s_{j}=\rho_{j} + \sum_{l=0}^{j-1}\mu_{j,l}\rho_{l}$ for a set of coefficients $\{\mu_{j, l}\}_{l=0}^{j-1}$.  

By normalizing the orthogonal polynomials $s_j$ we obtain the orthonormal OP denoted $\{S_{j}\}_{j=0}^{\infty}=\{d_{j}s_{j}\}_{j=0}^{\infty}$.
\end{deft}

\begin{rem}
Note that the above definitions will remain unchanged if in defining the symmetric polynomials $\rho_j$ we used $P_{j2}$ instead of $P_{j1}$; see \cite{nsty1} for details about this. 
\end{rem}

Observe that for each $j,  k \geq 0$ we have 
$\langle \rho_j , \rho_k \rangle = 6 \langle P_{j1}^{(0)} , P_{k1}^{(0)} \rangle,$ which implies that for $j\geq 0$, $\nm{\rho_{j}}_{L^{2}}^{2}=d_{j}^{-2}=6\nm{P_{j1}}_{L^{2}}^{2}.$ We remark that the size of $\nm{s_{j}}_{L^{2}}$ was already estimated in Theorem~\ref{nmpk}.  Furthermore, one checks easily  that $d_0^{-2}=\nm{\rho_0}_{L^{2}}^{2}=6\nm{P_{01}}_{L^{2}}^{2}=6$ 

As in subsection~\ref{subsec3}, we now plot the symmetric orthonormal polynomials as well as some of the related sequence. In particular,  Figure~\ref{fig:sk},  displays plots of the sequences $\log (\nm{s_{k}}_{L^{2}}^{-2})=\log (d_{k}^{2})$, $\log (-b_k)$, and $\log c_{k}$ for these symmetric OP. The behavior of the coefficients $c_k, b_k,$ and $d_k$ in the case of the symmetric OP $S_k$ are very similar to those we observed for the antisymmetric OP.

\begin{figure}[htp]
\begin{center}
\includegraphics[scale=.5]{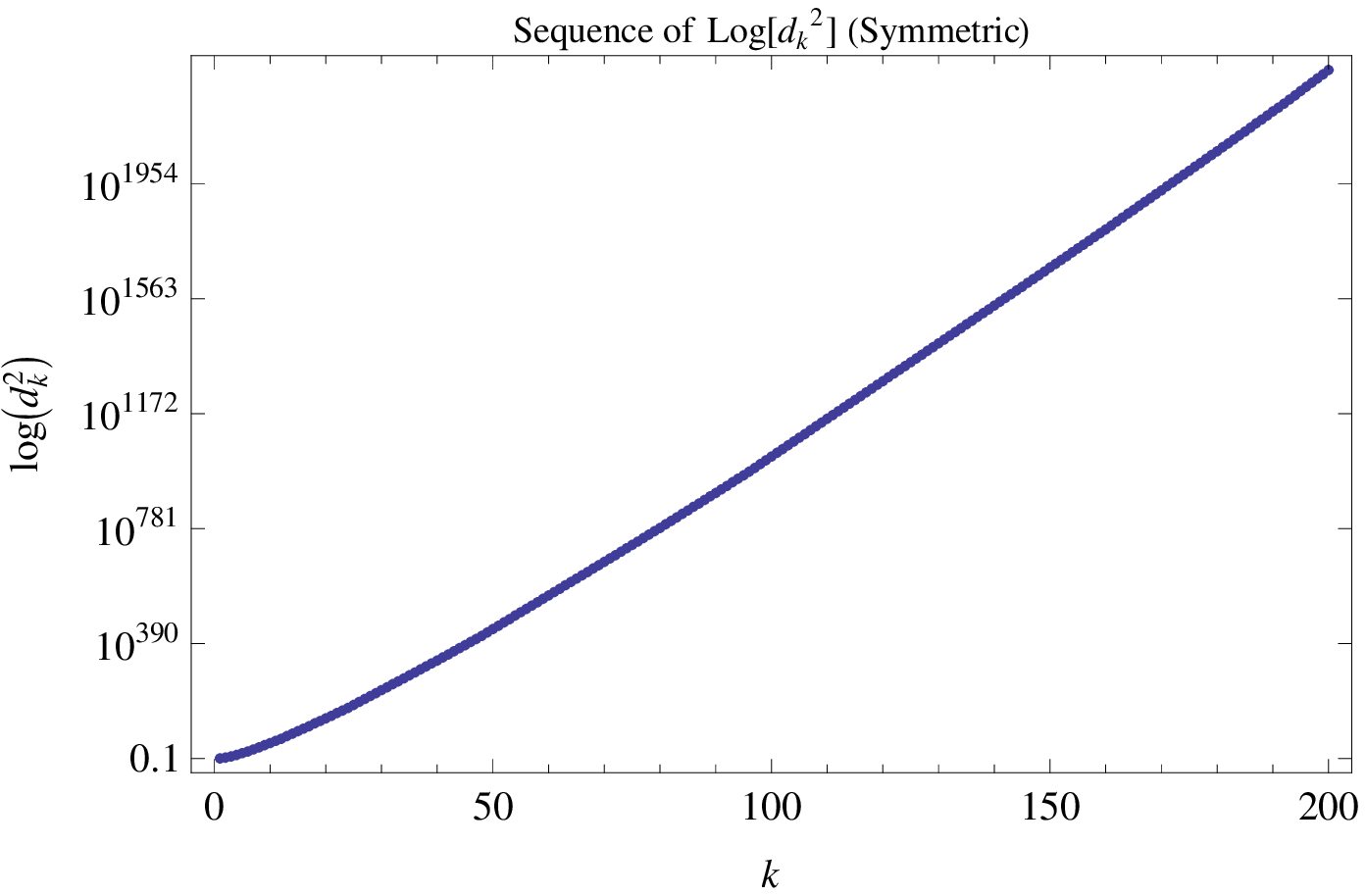} 
\includegraphics[scale=.5]{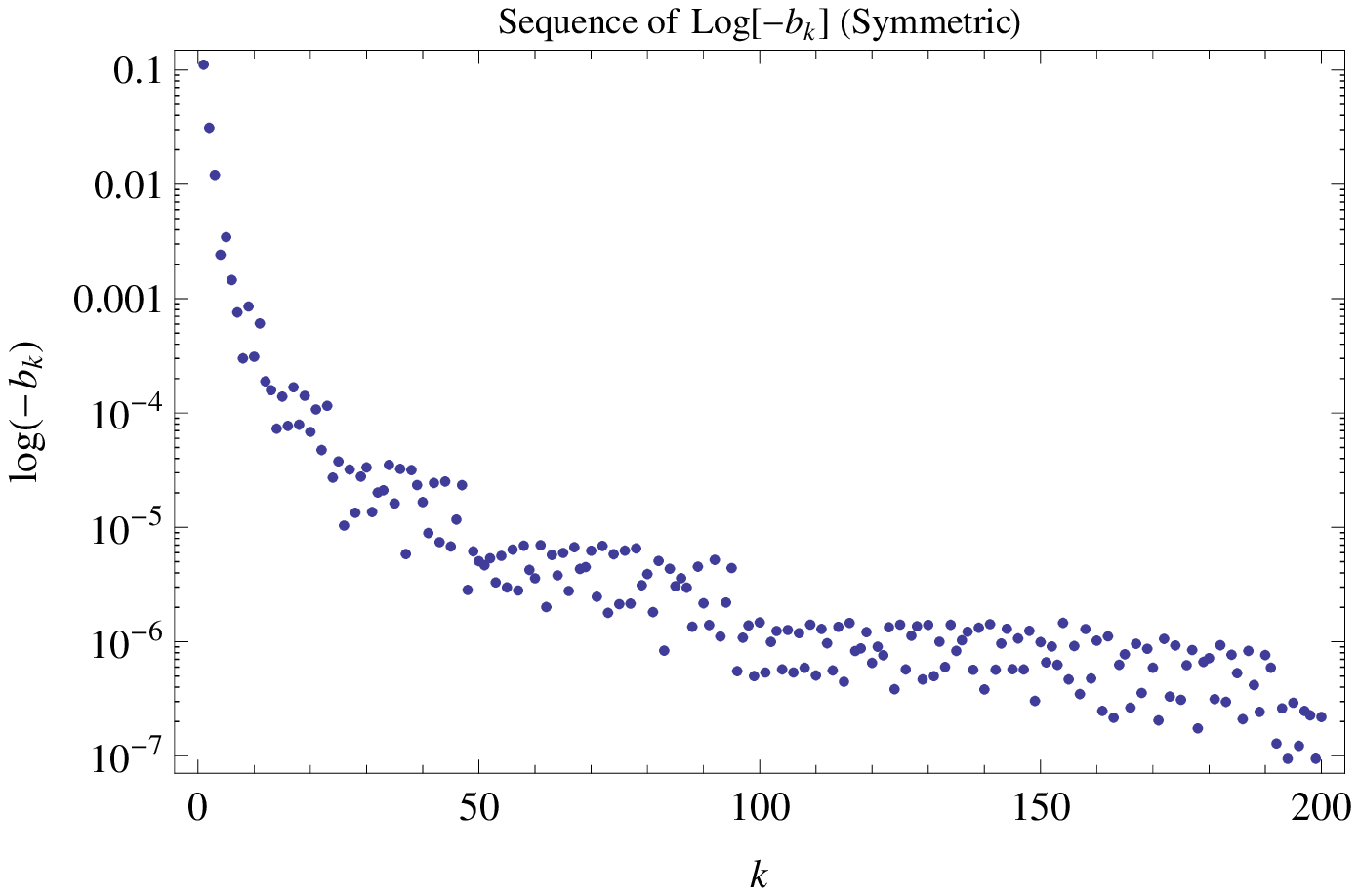} 
\includegraphics[scale=.5]{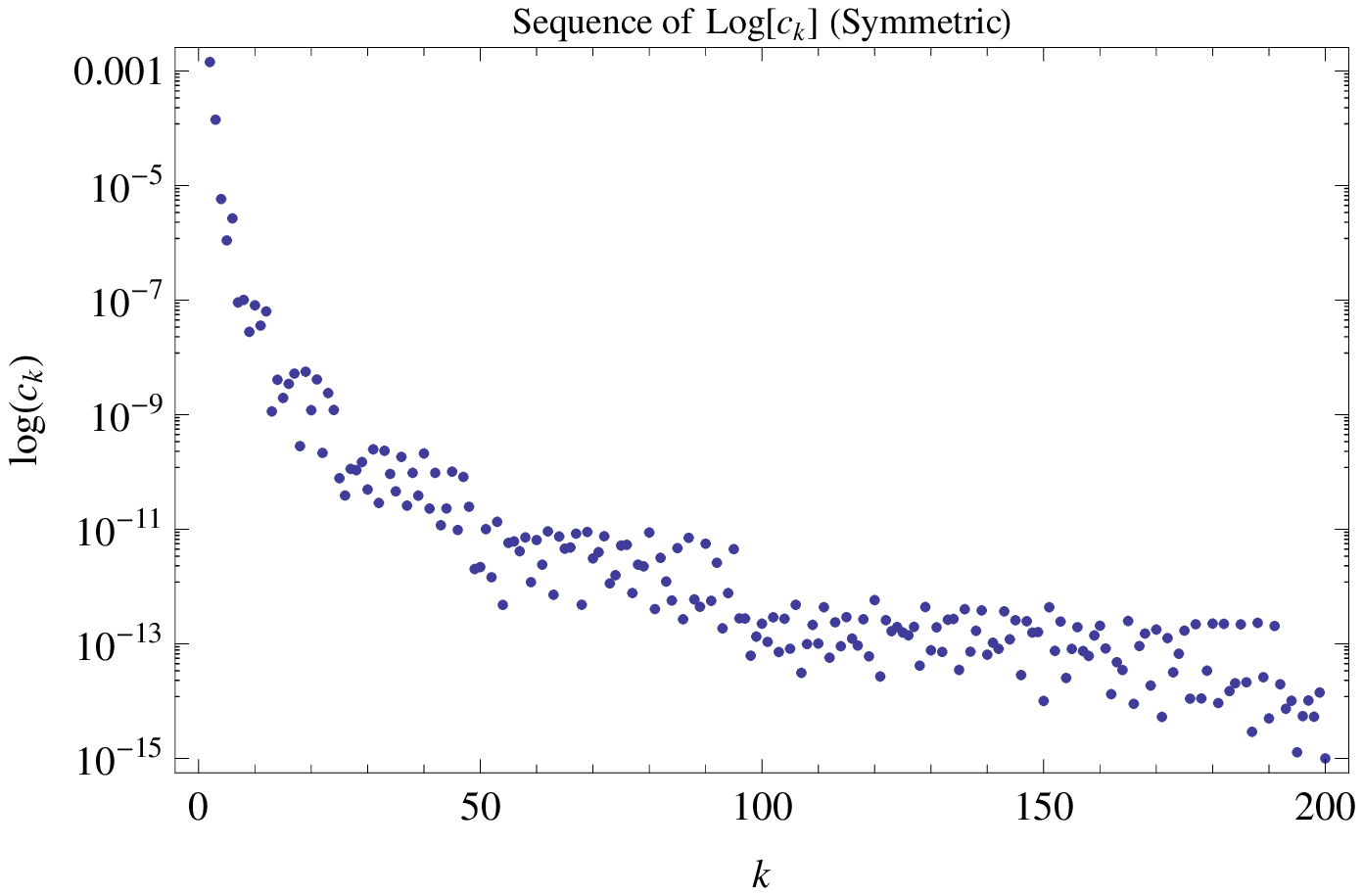} 
\caption{Plots of the $\log d_k^2$, $\log(-b_k)$, and $\log c_k$ for the symmetric OP $s_k$ in log scale.}
\label{fig:sk}
\end{center}
\end{figure}

In Table \ref{tab:Sj}, we show the coefficients $\mu_{j,\ell}$ for $s_j$, when $j=0, 1, \hdots, 7$

\begin{table}[htbp]
  \centering
  \caption{$s_j$ and their coefficients $\mu_{j,\ell} $}
    \label{tab:Sj}
    \begin{tabular}{rrrrrrrr}
          & $\rho_0$ & $\rho_1$ & $\rho_2$ & $\rho_3$ & $\rho_4$ & $\rho_5$ & $\rho_6$ \\
    $s_0$    & 1     & 0     & 0     & 0     & 0     & 0     & 0 \\
    $s_1$    & -5.56E-02 & 1     & 0     & 0     & 0     & 0     & 0 \\
    $s_2$    & 5.34E-04 & -2.44E-02 & 1     & 0     & 0     & 0     & 0 \\
    $s_3$    & -6.98E-07 & 9.72E-05 & -1.23E-02 & 1     & 0     & 0     & 0 \\
    $s_4$    & 1.56E-09 & -3.20E-07 & 6.15E-05 & -9.93E-03 & 1     & 0     & 0 \\
    $s_5$    & -1.46E-12 & 3.54E-10 & -9.41E-08 & 2.62E-05 & -6.48E-03 & 1     & 0 \\
    $s_6$    & 1.44E-16 & -8.51E-14 & 5.21E-11 & -2.94E-08 & 1.41E-05 & -5.02E-03 & 1 \\
    \end{tabular}
\end{table}

Four of the symmetric orthogonal polynomials $S_j=d_j s_j$, corresponding to $j=0, 3, 5,$ and $6$ are shown in Figures~\ref{fig:6sop}. As observed for the antisymmetric OP, $|S_{k}(q_{i})|$, which is constant for $i=1, 2, 3$ (due to symmetry) is very large compare to the values of $S_k$ at non-boundary points in $SG$.

\begin{figure}[hbp]
\begin{center}
\includegraphics[scale=.33]{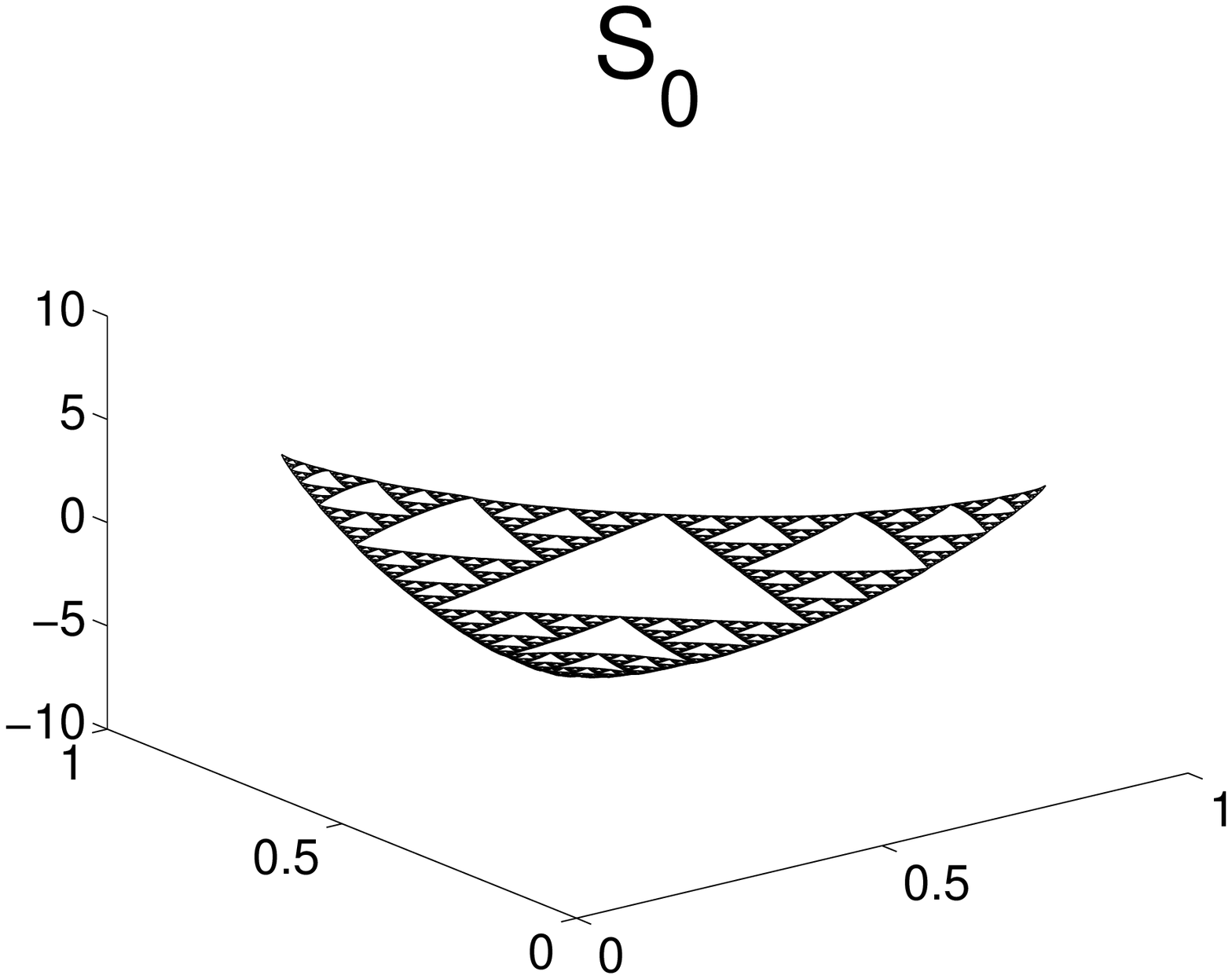} 
\includegraphics[scale=.33]{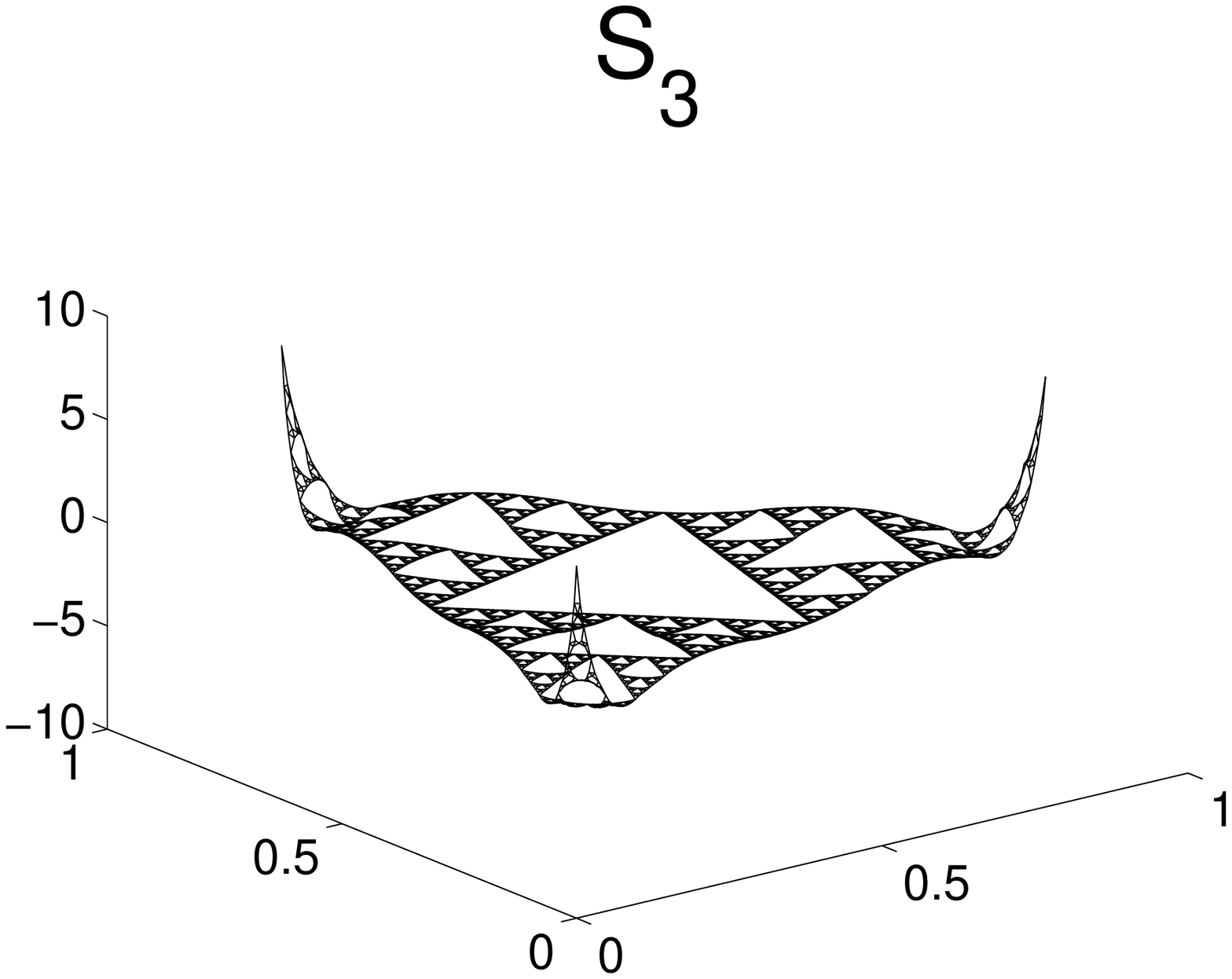}
\includegraphics[scale=.33]{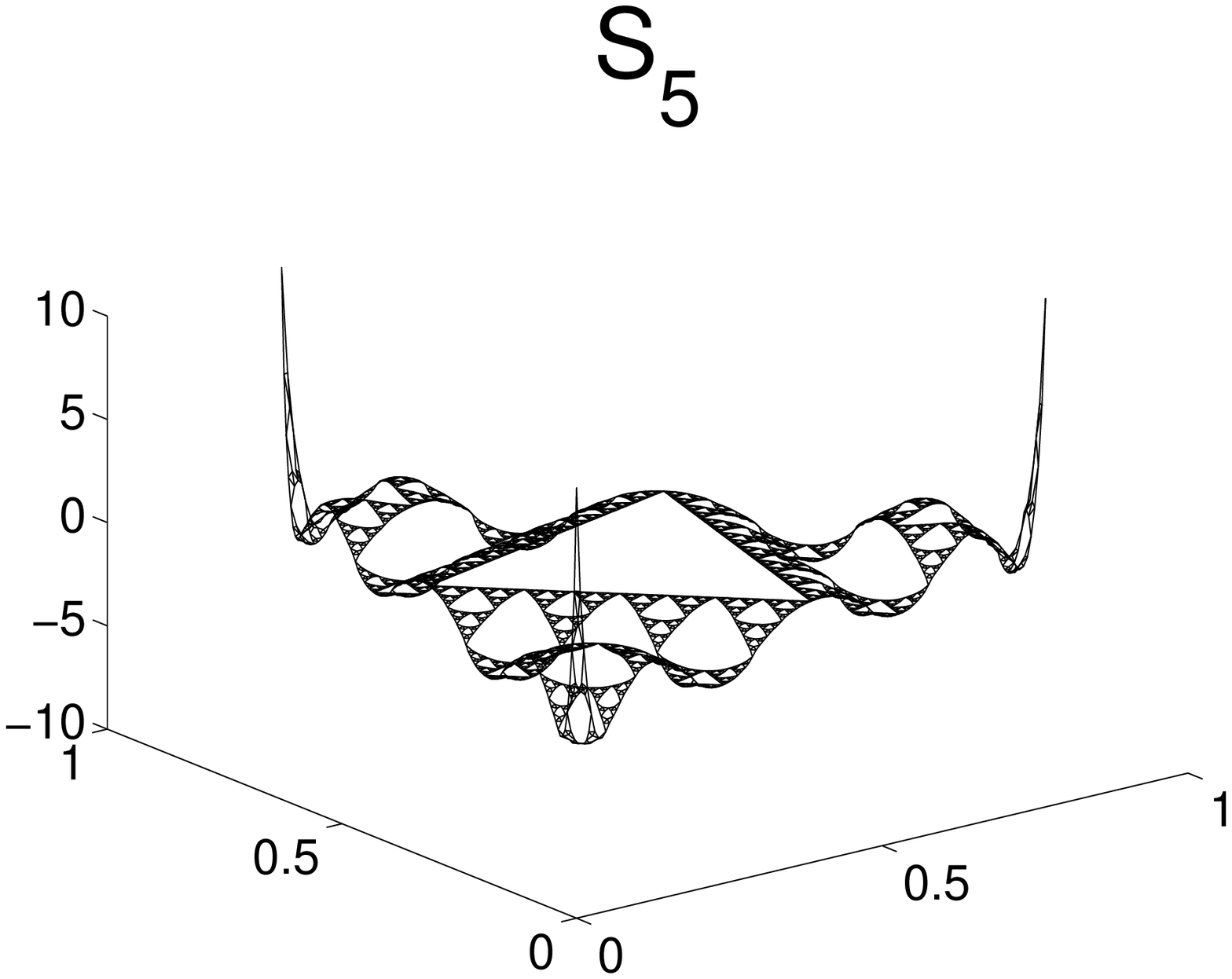} 
\includegraphics[scale=.33]{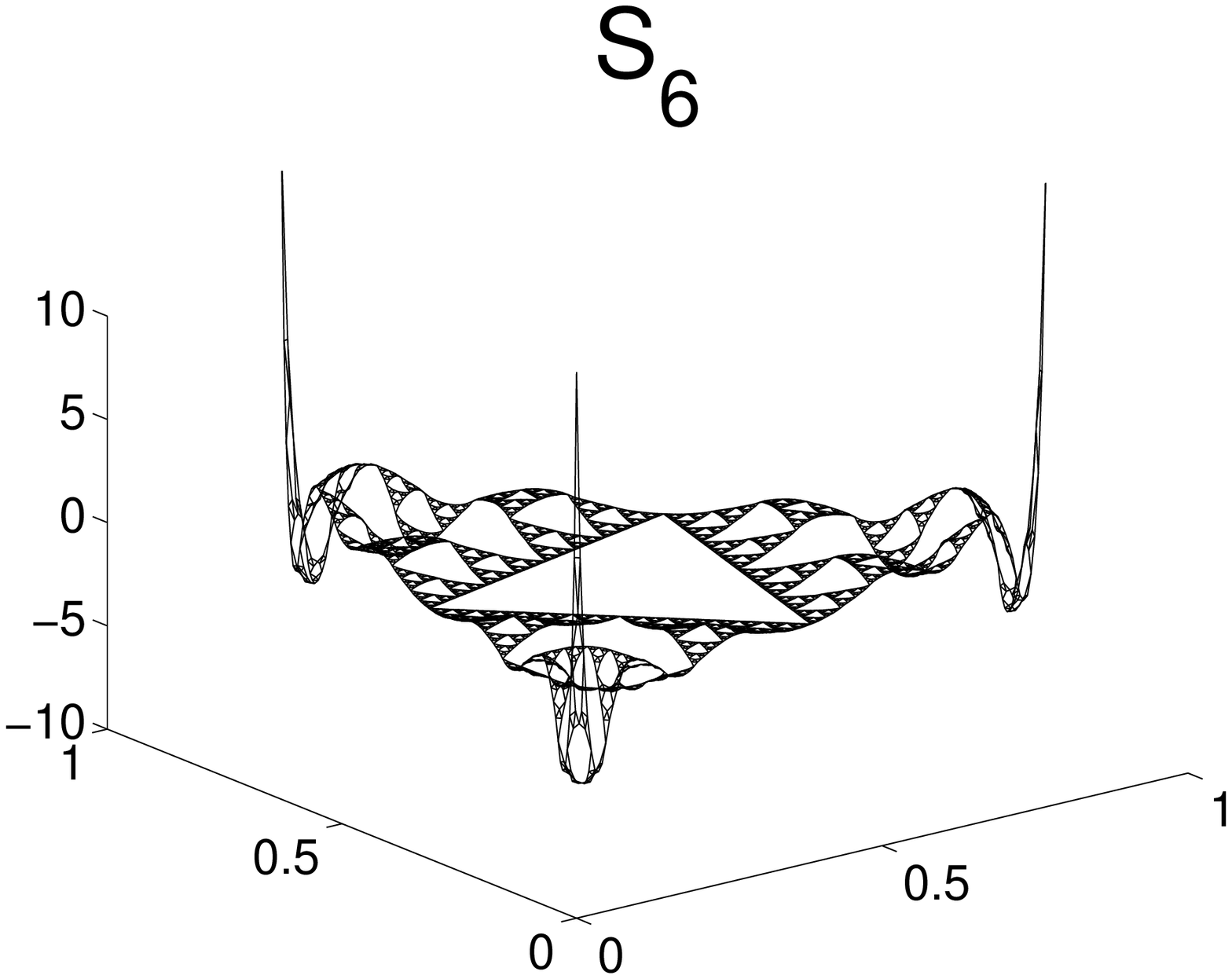}
\caption{4 symmetric orthonormal polynomials.}
\label{fig:6sop}
\end{center}
\end{figure}

\subsection{Orthonormal system of polynomials}\label{subsec5}
By combining both the antisymmetric and symmetric OP, we can form an orthogonal system of polynomials in $L^{2}(SG)$. Notice, that this system will not span the whole space $L^2$, since the space of polynomials is not dense in $L^{2}(SG)$ either.

Using Lemma~\ref{lemma1} we can prove that the set $\{ P_{j3}^{(0)} , P_{j3}^{(1)} , P_{j3}^{(2)} \}$ forms a tight frame for its span $\mathcal{F}_{j} =span \{ P_{j3}^{(0)} , P_{j3}^{(1)} , P_{j3}^{(2)} \}$. That is, there exists a constant $A_{j}>0$ such that for each $P \in \mathcal{F}_{j} =span \{ P_{j3}^{(0)} , P_{j3}^{(1)} , P_{j3}^{(2)} \}$ we have $$P(x)=A_{j}\sum_{n=0}^{2}\ip{P}{P^{(n)}_{j3}}P^{(n)}_{j3}(x), \forall x \in SG.$$ For more on frame theory we refer to \cite{chr}.

\begin{thm}\label{tf}
For each $j\ge0$, the set $\big\{ P_{j3}^{(0)} , P_{j3}^{(1)} , P_{j3}^{(2)} \big\}$ forms a tight frame for its two-dimensional span $\mathcal{F}_{j} =span \{ P_{j3}^{(0)} , P_{j3}^{(1)} , P_{j3}^{(2)} \}$ with frame bound $A_j=\frac{3}{2} \|P_{j3}^{(0)}\|_{L^{2}}^2$.
\end{thm}

\begin{proof}
The frame bounds of $\big\{ P_{j3}^{(0)} , P_{j3}^{(1)} , P_{j3}^{(2)} \big\}$ are determined by the eigenvalues of its Gram matrix $G$ given by 
\[
\begin{array}{ccc}
G= & 
\begin{pmatrix} P_{j3}^{(0)} \\ P_{j3}^{(1)} \\ P_{j3}^{(2)} \end{pmatrix}
\begin{pmatrix} P_{j3}^{(0)} & P_{j3}^{(1)} & P_{j3}^{(2)} \end{pmatrix} &
= \dfrac{1}{2} \|P_{j3}^{(0)}\|_{L^{2}}^2 \begin{pmatrix} 2 & -1 & -1 \\ -1 & 2 & -1 \\ -1 & -1 & 2 \end{pmatrix}
\end{array}
\]
G has eigenvalues $\{ 0, \frac{3}{2} \|P_{j3}^{(0)}\|_{L^{2}}^2, \frac{3}{2} \|P_{j3}^{(0)}\|_{L^{2}}^2 \}$. This automatically shows that the frame operator $$S=\begin{pmatrix} P_{j3}^{(0)} & P_{j3}^{(1)} & P_{j3}^{(2)} \end{pmatrix} \begin{pmatrix} P_{j3}^{(0)} \\ P_{j3}^{(1)} \\ P_{j3}^{(2)} \end{pmatrix}$$ is the $2\times 2$ diagonal matrix with $\frac{3}{2} \|P_{j3}^{(0)}\|_{L^{2}}^2$ on the diagonal. Thus,  $\big\{ P_{j3}^{(0)} , P_{j3}^{(1)} , P_{j3}^{(2)} \big\}$ is a tight frame with frame bound $A_j = \frac{3}{2} \|P_{j3}^{(0)}\|_{L^{2}}^2$
\end{proof}

Recall that $\{S_{j}\}_{j=0}^{\infty}$ are the symmetric orthonormal polynomials, and for each $i=0, 1, 2$, corresponds a family of antisymmetric orthonormal polynomials $\{Q_{j}^{(i)}\}_{j=0}^{\infty}$. With these notations, the following result holds.

\begin{thm}\label{opcollection}
The set
\begin{equation*}
\left \{\phi_j^{(i)}, i=0, 1,2\right \}_{j\geq 0} =: \left \{  \sqrt{\frac{2}{3}} Q_j^{(i)} + \sqrt{\frac{1}{3}}S_j, i=0, 1, 2    \right \}_{j=0}^{\infty}
\end{equation*}
is an orthonormal basis for a subspace $\mathcal{P}$ of $L^2(SG)$.
\end{thm}

\begin{proof}
For $j\geq 0$ and $i \in \{0,1,2\}$, let an element in the above system be defined by 
\begin{equation*}
\phi_j^{(i)} = \sqrt{\frac{2}{3}} Q_j^{(i)} + \sqrt{\frac{1}{3}}S_j
\end{equation*}
Then, the inner product of two orthonormal system elements is: 
\begin{eqnarray*}
\langle \phi_j^{(i)}, \phi_k^{(\ell)} \rangle &=& \langle \sqrt{\frac{2}{3}} Q_j^{(i)} + \sqrt{\frac{1}{3}}S_j, \sqrt{\frac{2}{3}} Q_k^{(\ell)} + \sqrt{\frac{1}{3}}S_k \rangle \\
&=& \frac{2}{3} \langle Q_j^{(i)}, Q_k^{(\ell)} \rangle + \frac{\sqrt{2}}{3} \langle S_j, Q_k^{(\ell)} \rangle + \frac{\sqrt{2}}{3} \langle Q_j^{(i)}, S_k \rangle + \frac{1}{3} \langle S_j, S_k \rangle .
\end{eqnarray*}
Recall that $S_j$ is fully symmetric and $Q_j^{(n)}$ is antisymmetric, so their inner product is zero for all $j$ and $n$.
Moreover, if $n\neq n'$, by the Gauss Green formula, we can rewrite 
\begin{equation*}
\langle Q_{k}^{(n)}, Q_{k}^{(n')} \rangle  = \sum _{j = 0}^{k} \sum_{i=0}^{2}  Q_{j}^{(n)}(q_i) \partial_n Q_{j}^{(n')} (q_i) - Q_{j}^{(n')}(q_i) \partial_n Q_{j}^{(n)} (q_i).
\end{equation*}
We now evaluate the right-hand side of the last equation. When $i = n$ or $i=n'$,  $Q_{j}^{(n)}(q_i) = 0$, and this eliminates two terms for each $j$ .
Let $a_j = Q_{j}^{(n)}(q_{n+1})$, then $-a_j = Q_{j}^{(n)}(q_{n-1})$ (by symmetry). If we let $b_j = \partial_n Q_{j}^{(n)}(q_{n+1})$, then $-b_j = \partial_n Q_{j}^{(n)}(q_{n-1})$.  If we let $c_j = \partial_n Q_{j}^{(n)}(q_{n})$, then  $c_j = 0$ since $\partial_n P_{j3}^{(n)}(q_{n})=0$ for all j.
Then we have
\begin{equation*}
\langle Q_{k}^{(n)}, Q_{k}^{(n')} \rangle  = \sum _{j = 0}^{k}   [0 - (- a_j)(c_j) + (a_j)(-b_j) - (-a_j)(b_j) + (a_j) (c_j) - 0]=2 a_j c_j = 0
\end{equation*}
More generally, the same argument can be used to prove that, $\langle Q_{k}^{(n)}, Q_{k'}^{(n')}\rangle  =\delta_{k, k'}\delta_{n, n'}$. Consequently, 

$$
\langle \phi_j^{(i)}, \phi_k^{(\ell)} \rangle= \frac{2}{3} \delta_{j,k}\delta_{i,\ell} + \frac{1}{3} \delta_{j,k}  = \delta_{j,k}\delta_{i,\ell}
$$
\end{proof}

\begin{rem} Notice that using a different normalization we can show that the set
\begin{equation*}
 \bigg\{ \sqrt{\frac{2}{3}} \sqrt{\frac{d_j}{\tilde{d_j}}} p_j^{(i)} + \sqrt{\frac{1}{3}} s_j , i=0, 1, 2\bigg\}_{j=0}^{\infty}
\end{equation*}
is also an orthogonal system for $L^2(SG)$.
\end{rem}
Four of these orthonormal polynomials are plotted in Figure~\ref{fig:6ONB}, and as might be expected, these OP do not seem to possess any obvious symmetry.

\begin{figure}[hbp]
\begin{center}
\includegraphics[scale=.33]{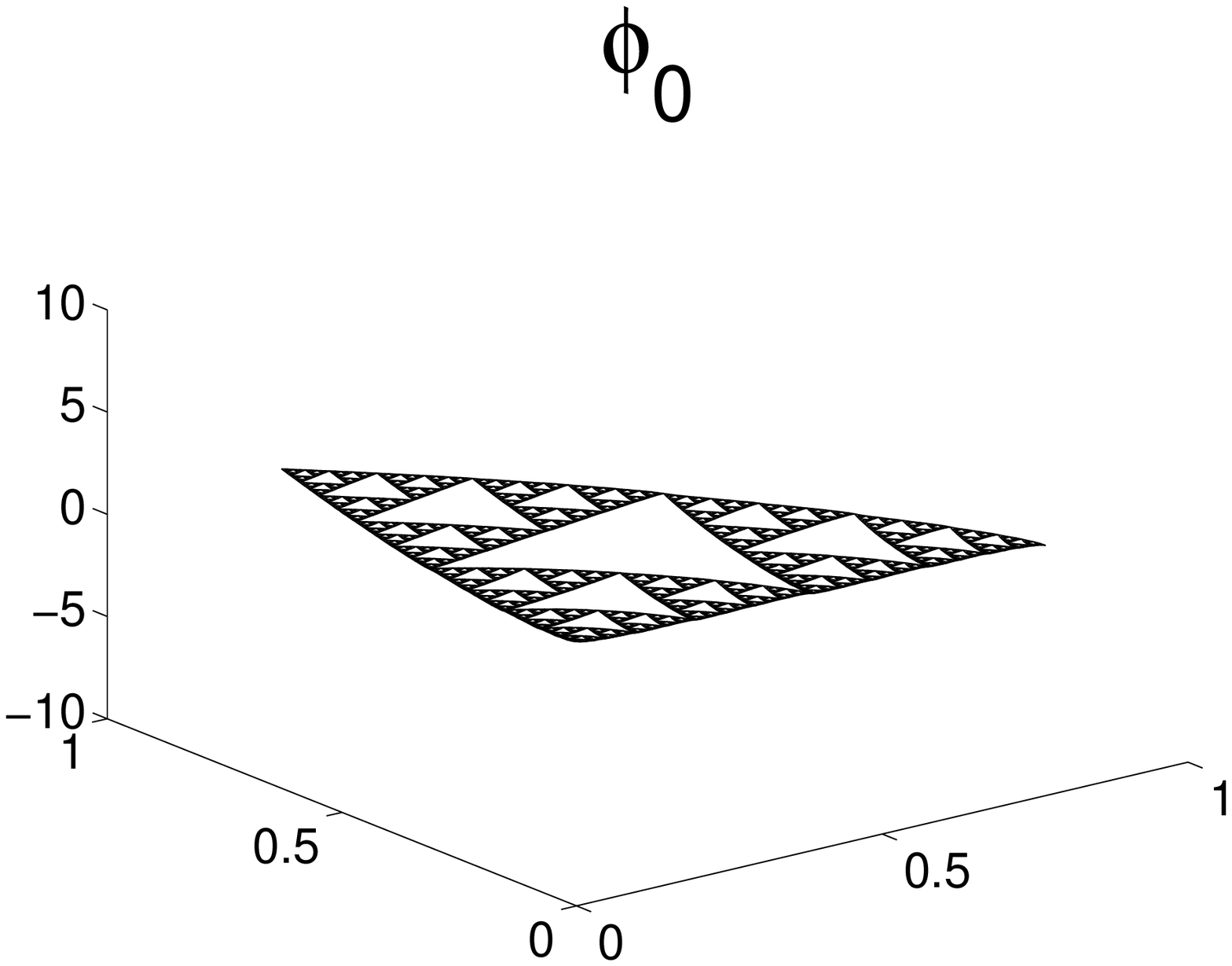} 
\includegraphics[scale=.33]{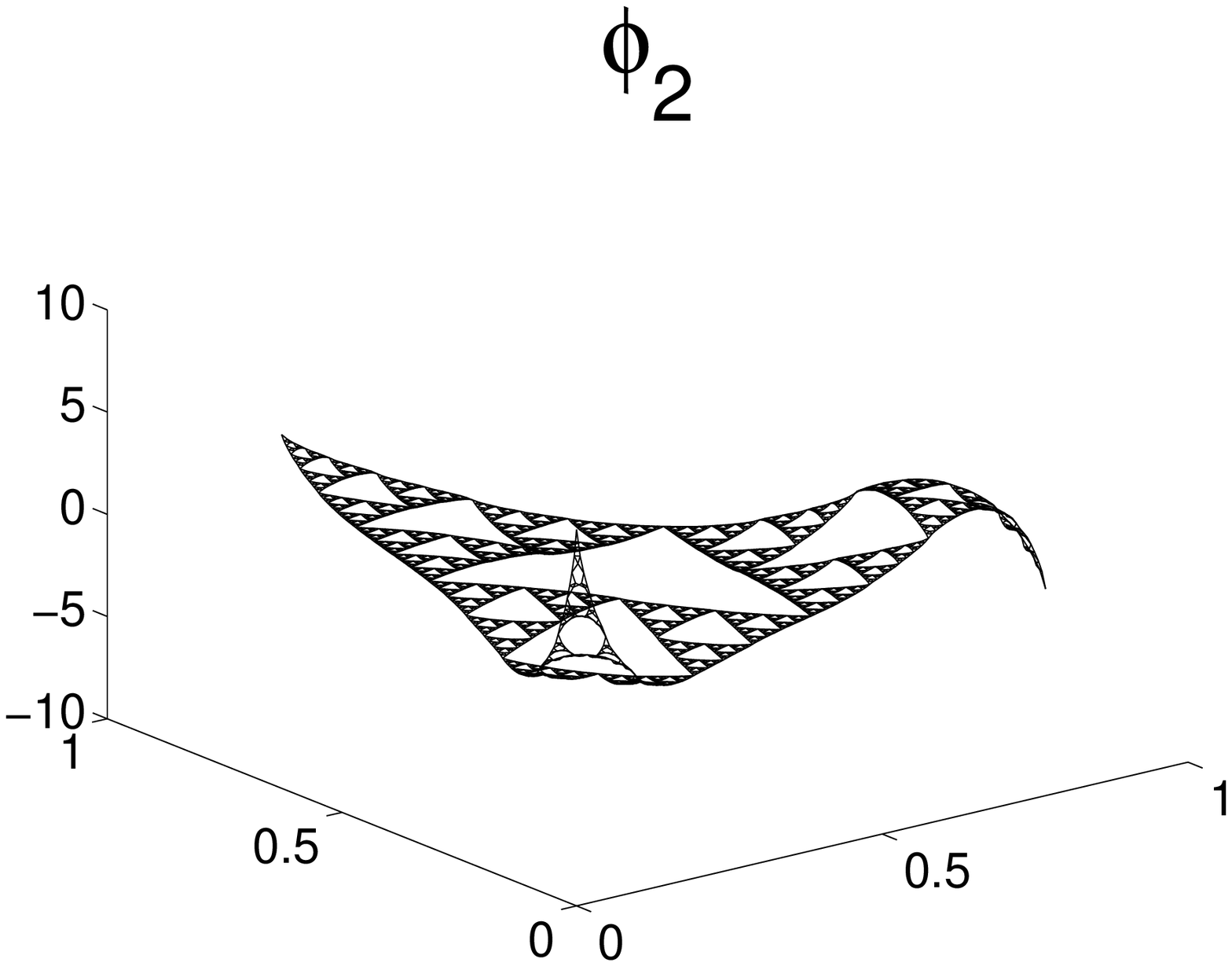}
\includegraphics[scale=.33]{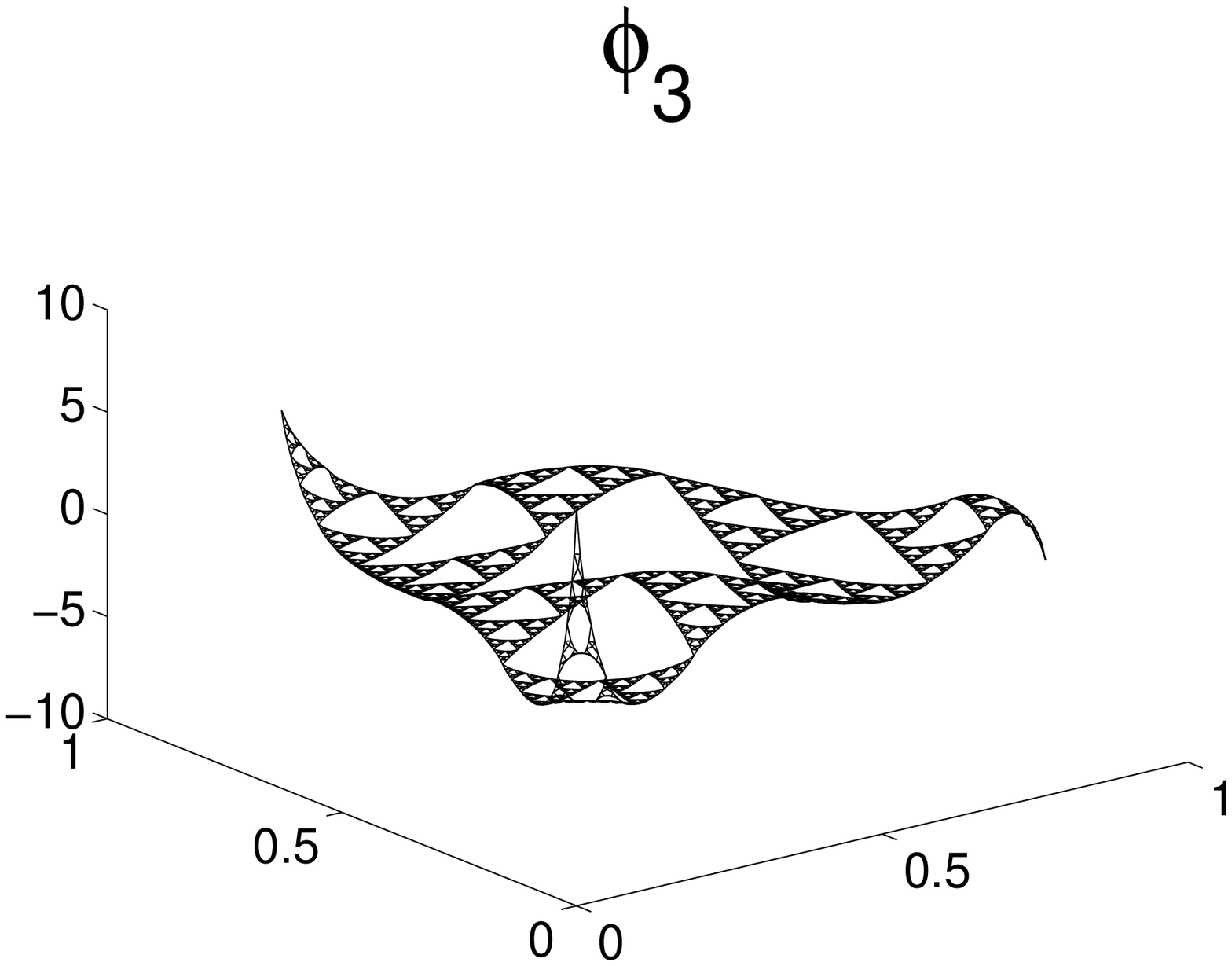} 
\includegraphics[scale=.33]{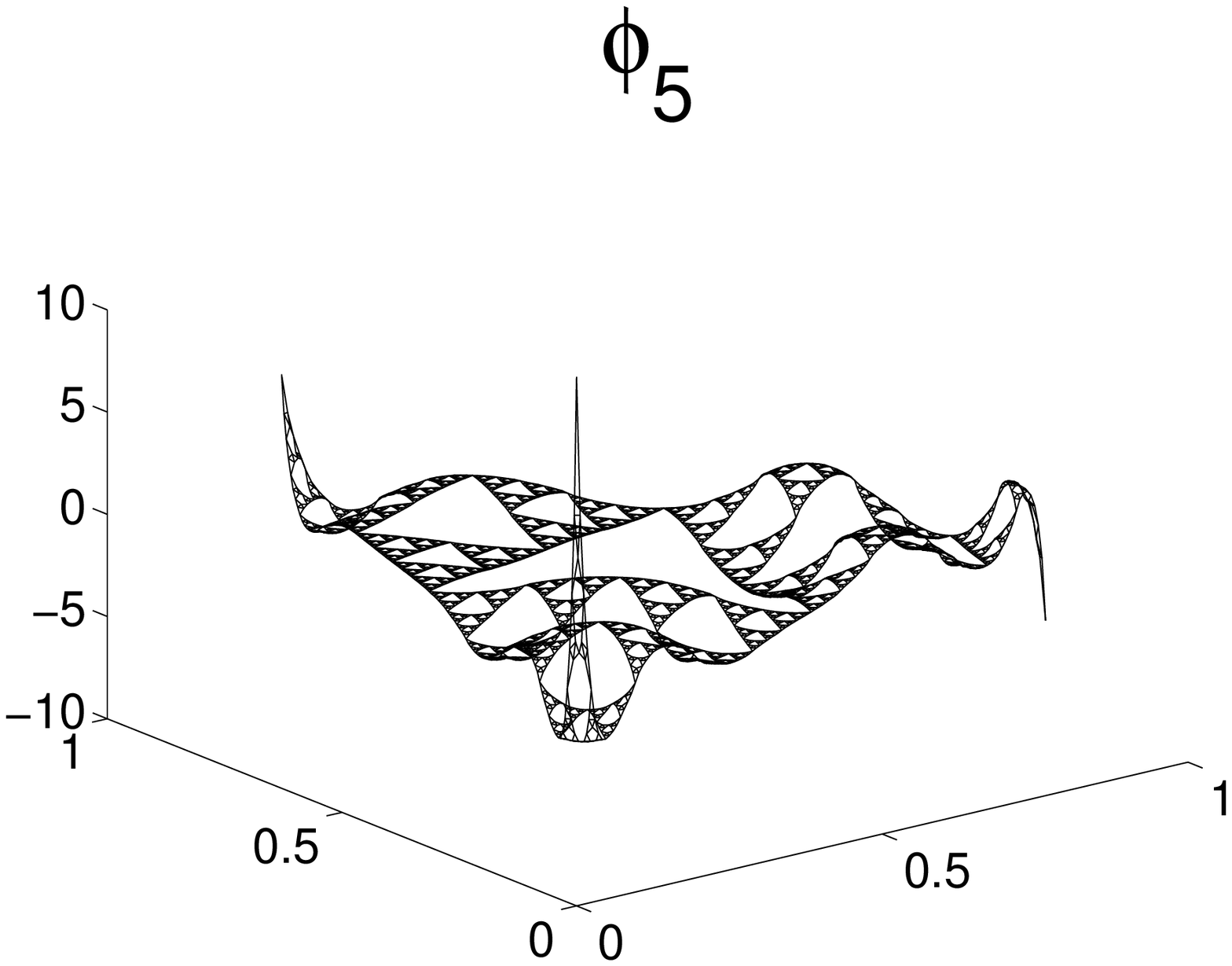}
\caption{4 orthonormal polynomials.}
\label{fig:6ONB}
\end{center}
\end{figure}

\subsection{Zero sets of orthogonal polynomials on $SG$}\label{subsec6}

Classical orthogonal polynomial theory suggests that zeros of higher order polynomials should interweave between the zeros of the next higher order polynomial.  Specifically, if $x_0$ is a zero of $p_{n}$, then there exist $x_a$ and $x_b$ zeros of $P_{n+1}$ such that $x_a < x_0 < x_b$. 

 On $SG$, we have not been able to establish a similar result. In fact, it is not clear what 'interweaving' will mean in the fractal setting. More generally, on $SG$, the zero sets of the OP seem very difficult to fully describe. However, using~\eqref{jacobieq} we see that an exact characterization of the zero set of $Q_n$ is the set of all points $x_0 \in SG$ such that $$\tilde{F}_{n}(x_{0})=J_{n}Q^{(n)}(x_{0}).$$ Though this looks similar to the characterization of the zeros of classical orthogonal polynomials in terms of the eigenvectors of a Jacobi matrix, there is a main difference in that $\tilde{F}_{n}(x_{0})$ is a function of the auxiliary polynomials $\tilde{f}_k$, $k-0, 1, \hdots, n-1.$ Despite this difficulty, we have some numerical data, suggesting some structures in these zero sets. In particular, we shall display below  nodal domains corresponding to some of these OP.

We first consider the zeros of the OP on the edges of $SG$. Notice
that there are two types of edges, which we call bottom edge (the edge
across from $q_0$) and side edges (edges which meet $q_0$).   We
present some numerical data in figures~\ref{fig:Qbottomedges},
~\ref{fig:Sbottomedges},  and~\ref{fig:edges}, which seem to indicate
that the zeros on these edges behave like zeros of classical
orthogonal polynomials. Because $S_k$ is fully
symmetric its behavior on side edges is the same as
their behavior on the bottom edge. The restriction of $Q_k$ to the side edge of $SG$ as illustrated in
Figure~\ref{fig:edges}.

\begin{figure}[htp]
\begin{center}
\includegraphics[scale=.2]{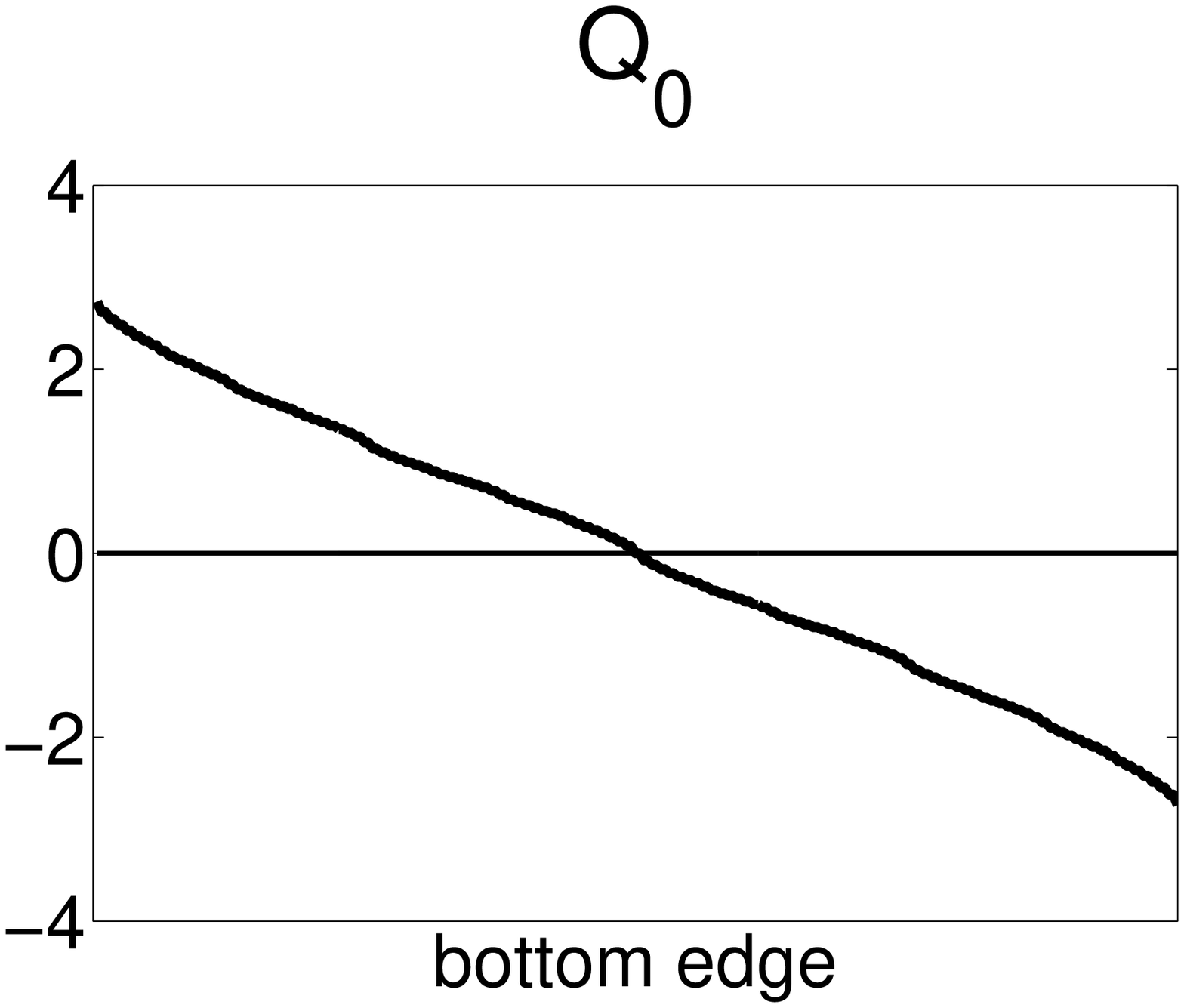}
\includegraphics[scale=.2]{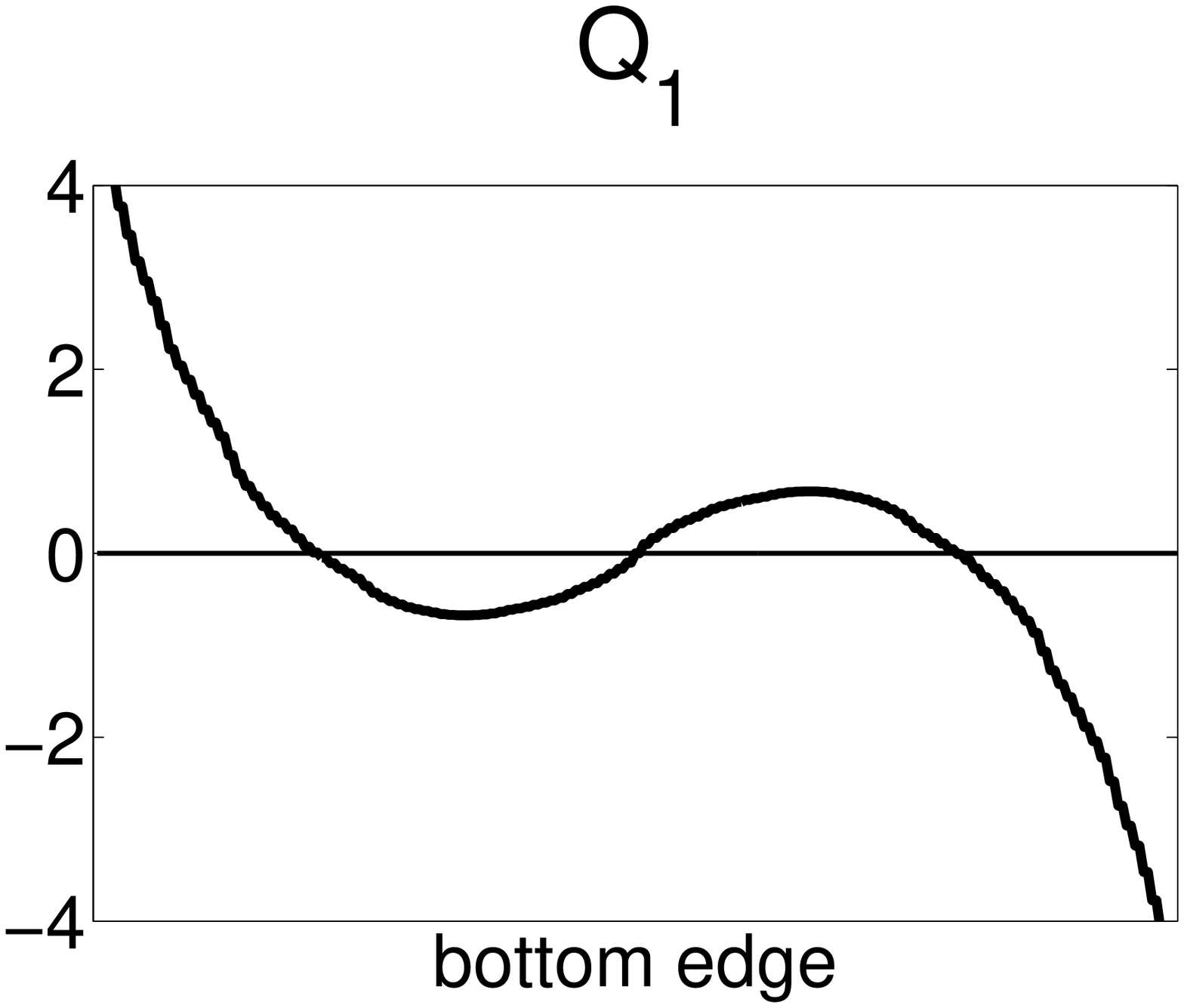}
\includegraphics[scale=.2]{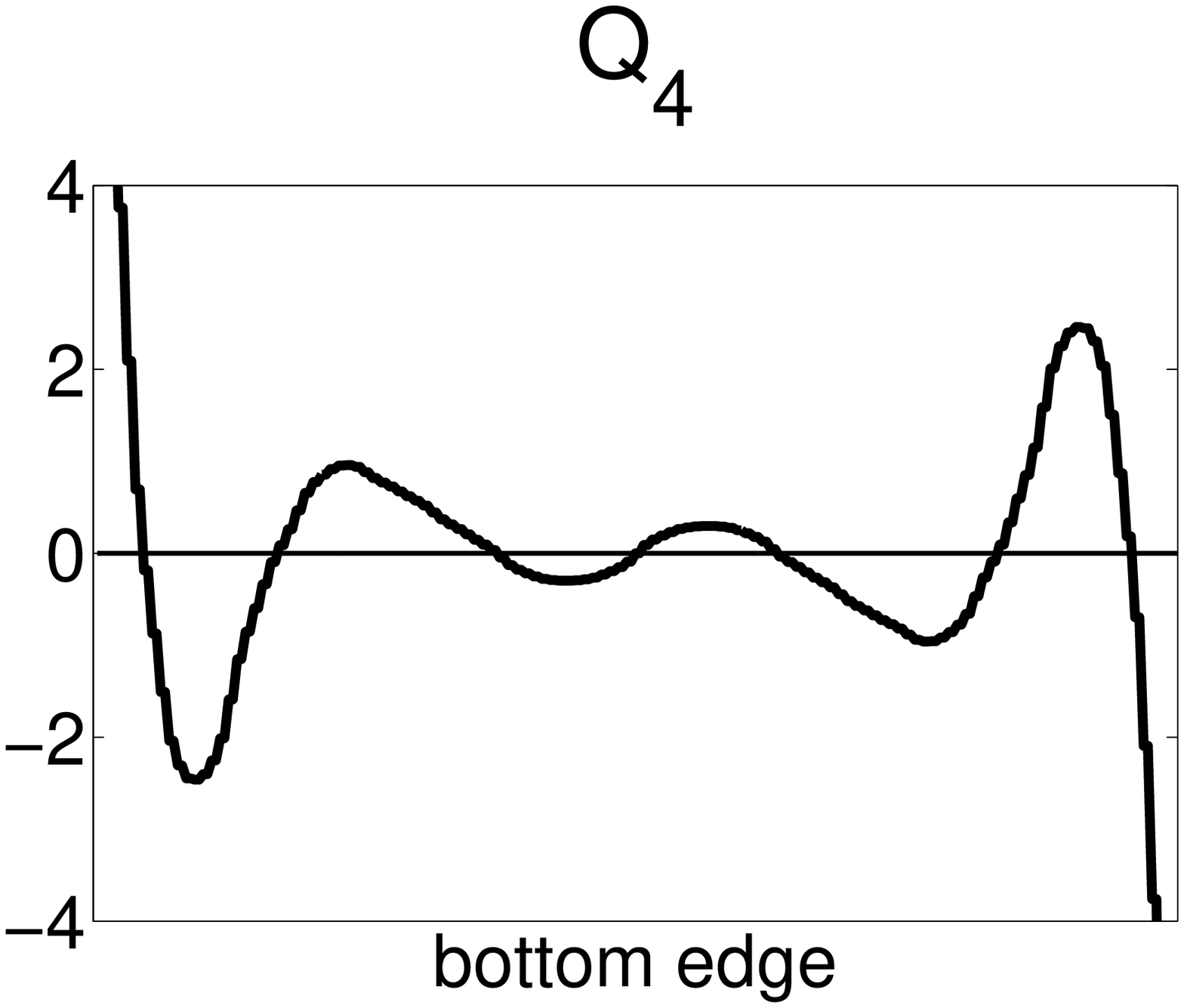}
\includegraphics[scale=.2]{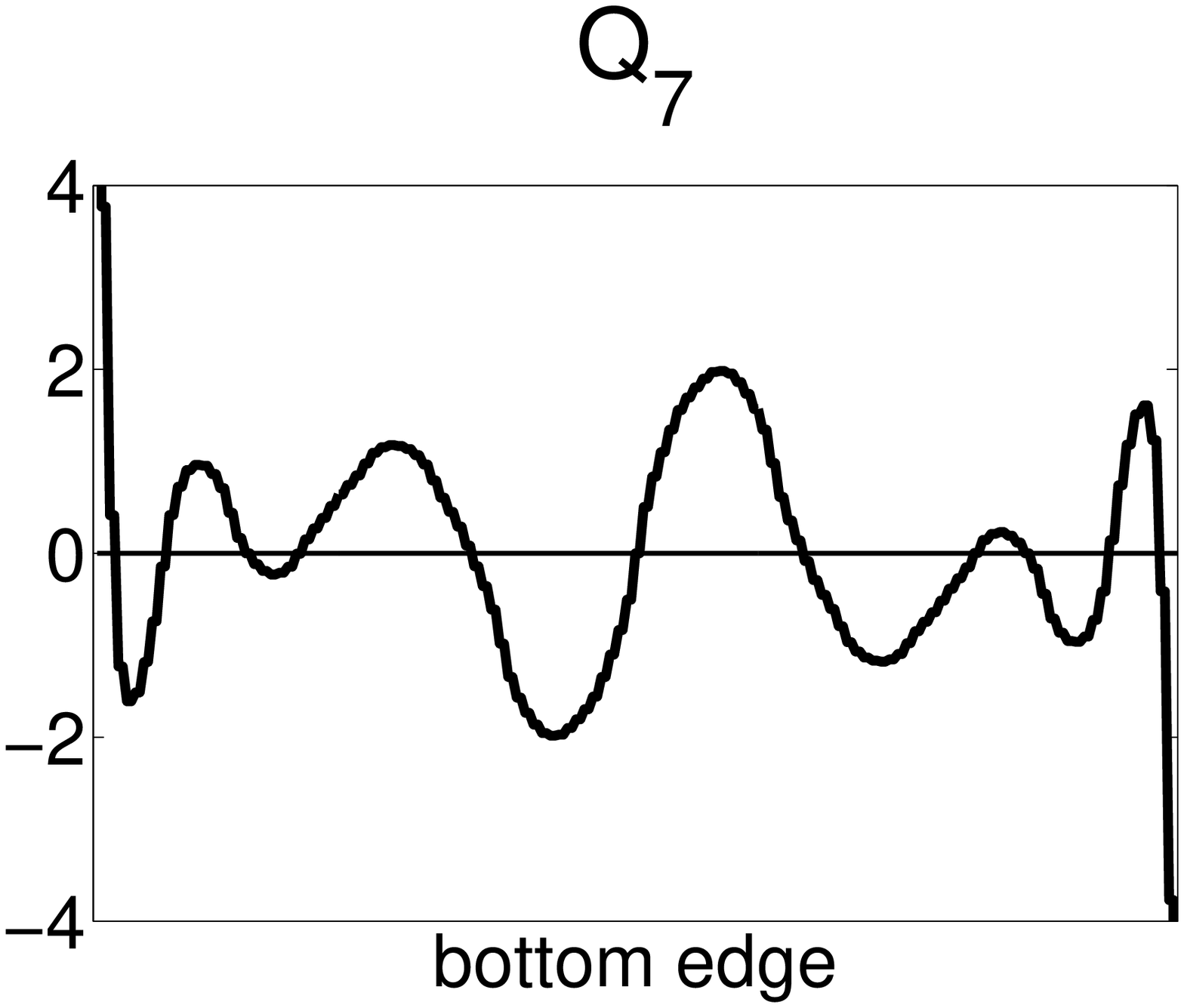}
\caption{Graph  $Q_k$, $k=0,1, 4, 7$ on the bottom edge of $SG$}
\label{fig:Qbottomedges}
\end{center}
\end{figure}

\begin{figure}[htp]
\begin{center}
\includegraphics[scale=.2]{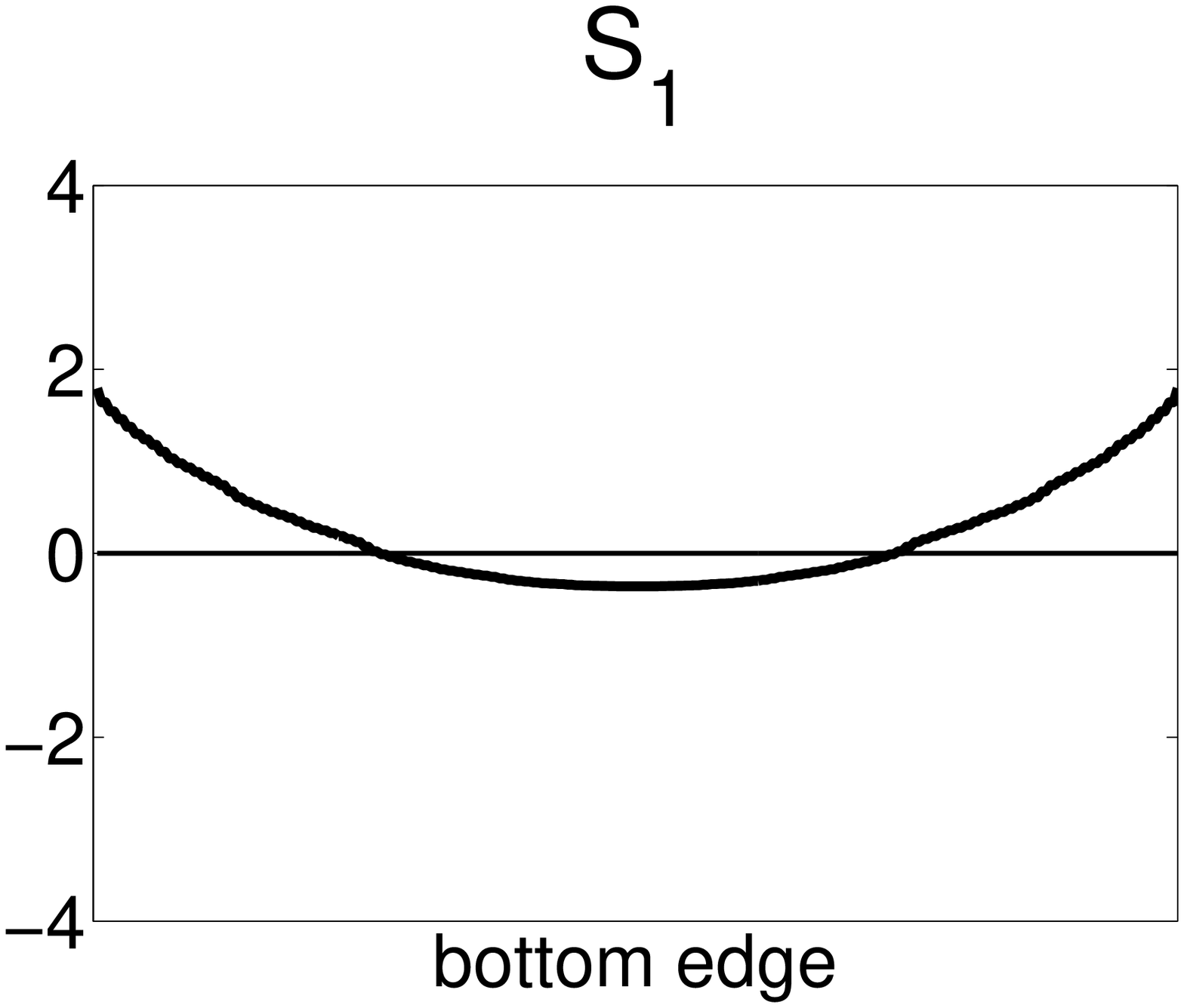}
\includegraphics[scale=.2]{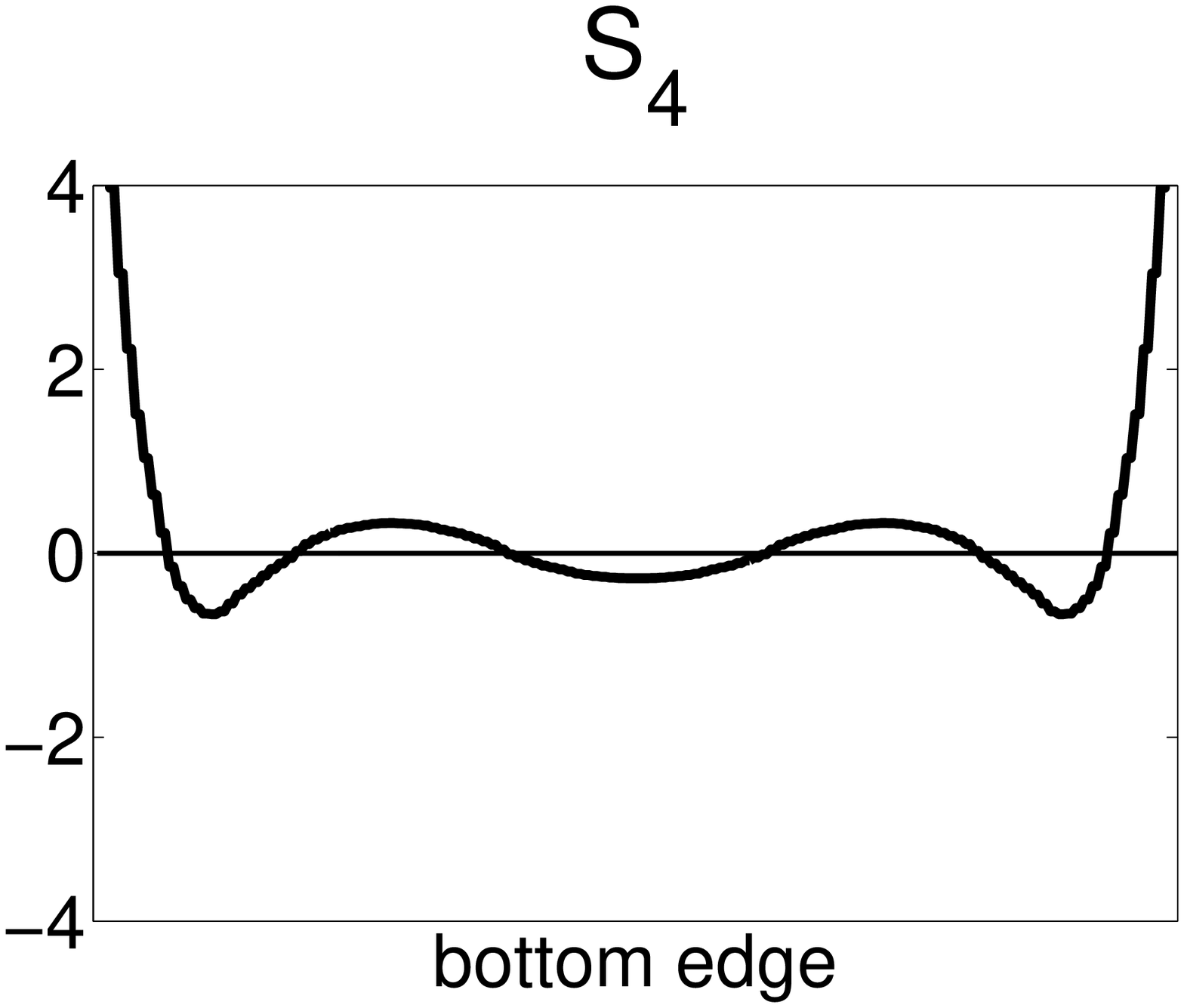}
\includegraphics[scale=.2]{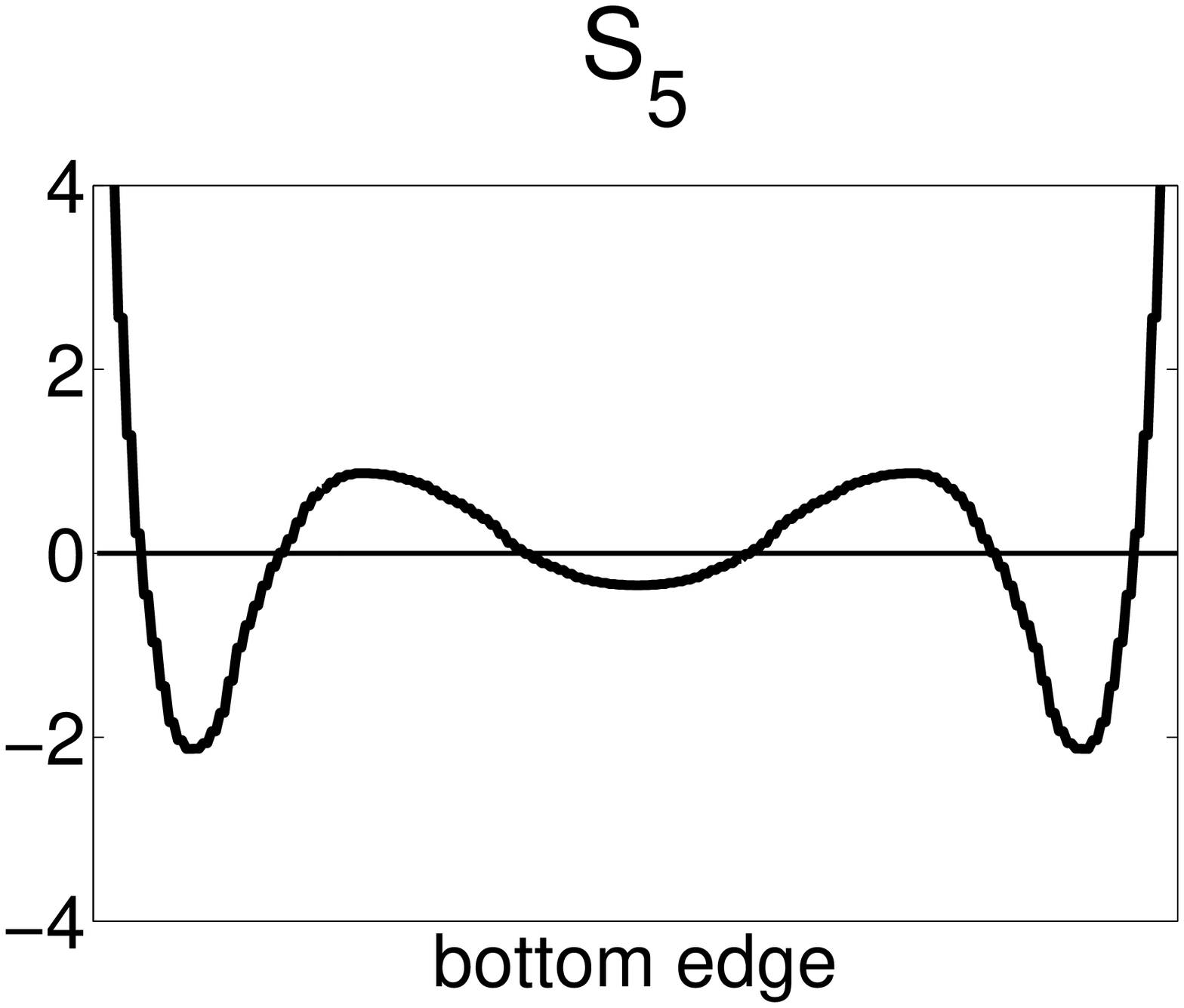}
\includegraphics[scale=.2]{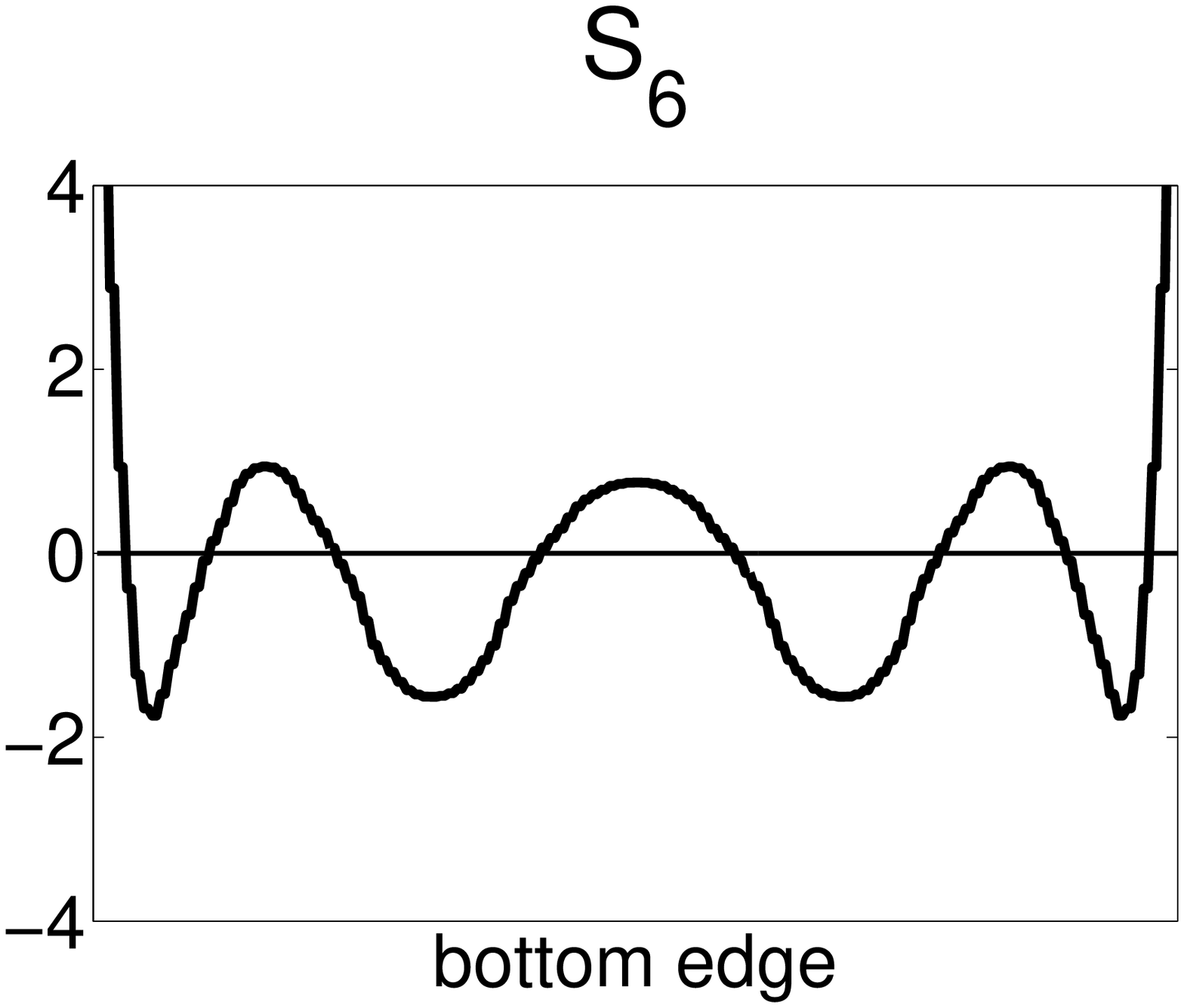}
\caption{Graph of $S_k$, $k=1, 4, 5, 6$ on the  bottom edge of $SG$}
\label{fig:Sbottomedges}
\end{center}
\end{figure}

\begin{figure}[hbp]
\begin{center}
\includegraphics[scale=.2]{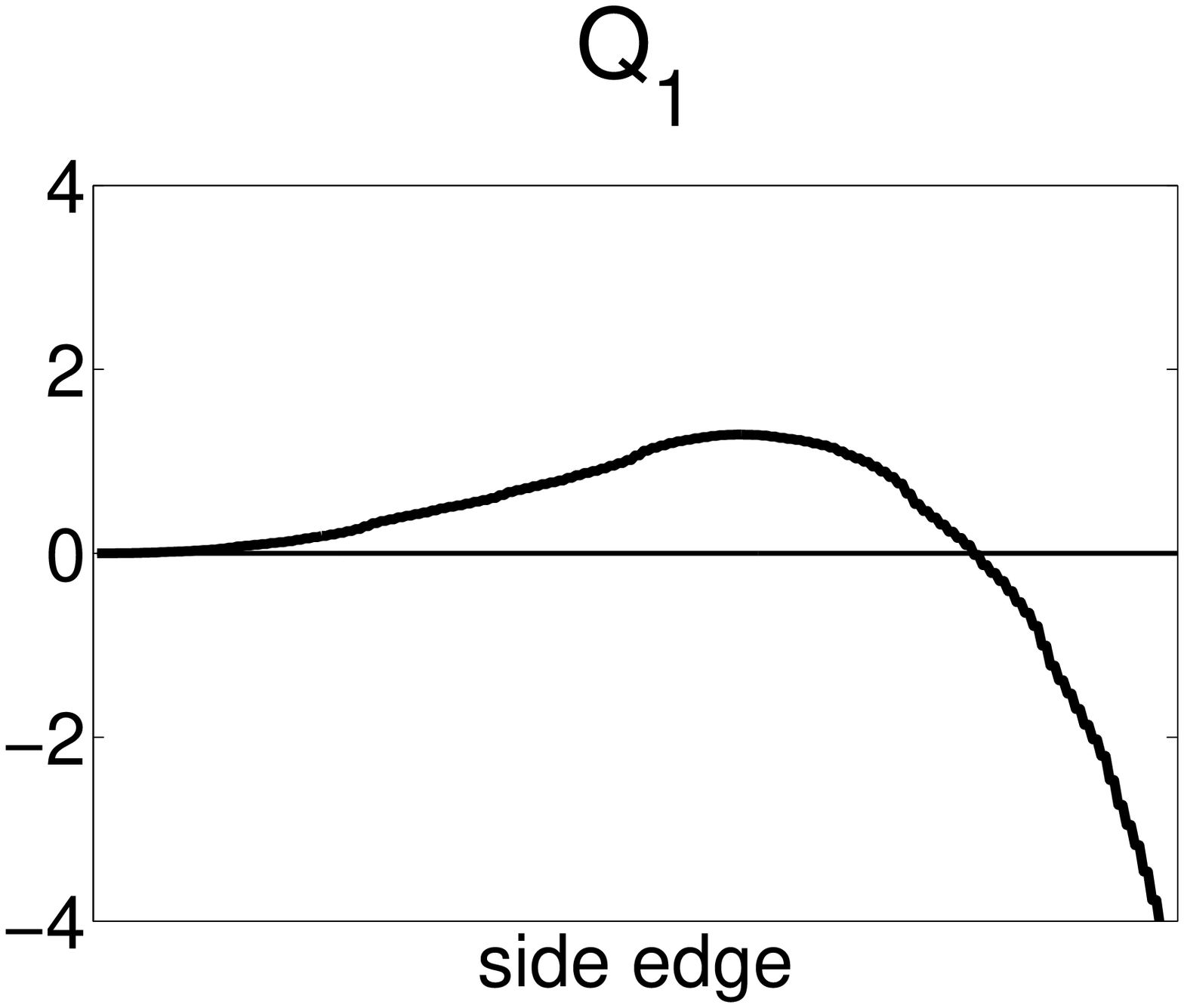}
\includegraphics[scale=.2]{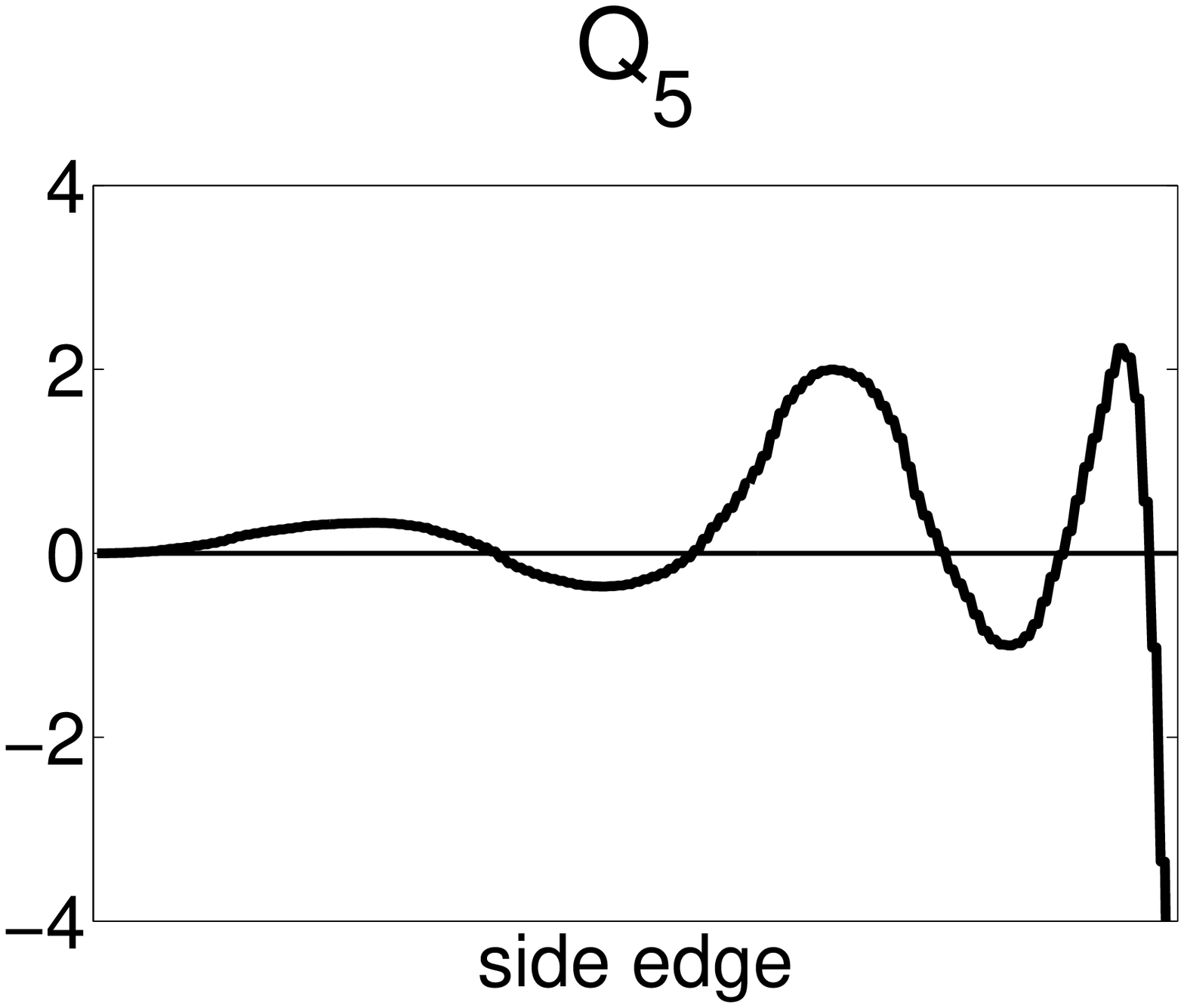}
\includegraphics[scale=.2]{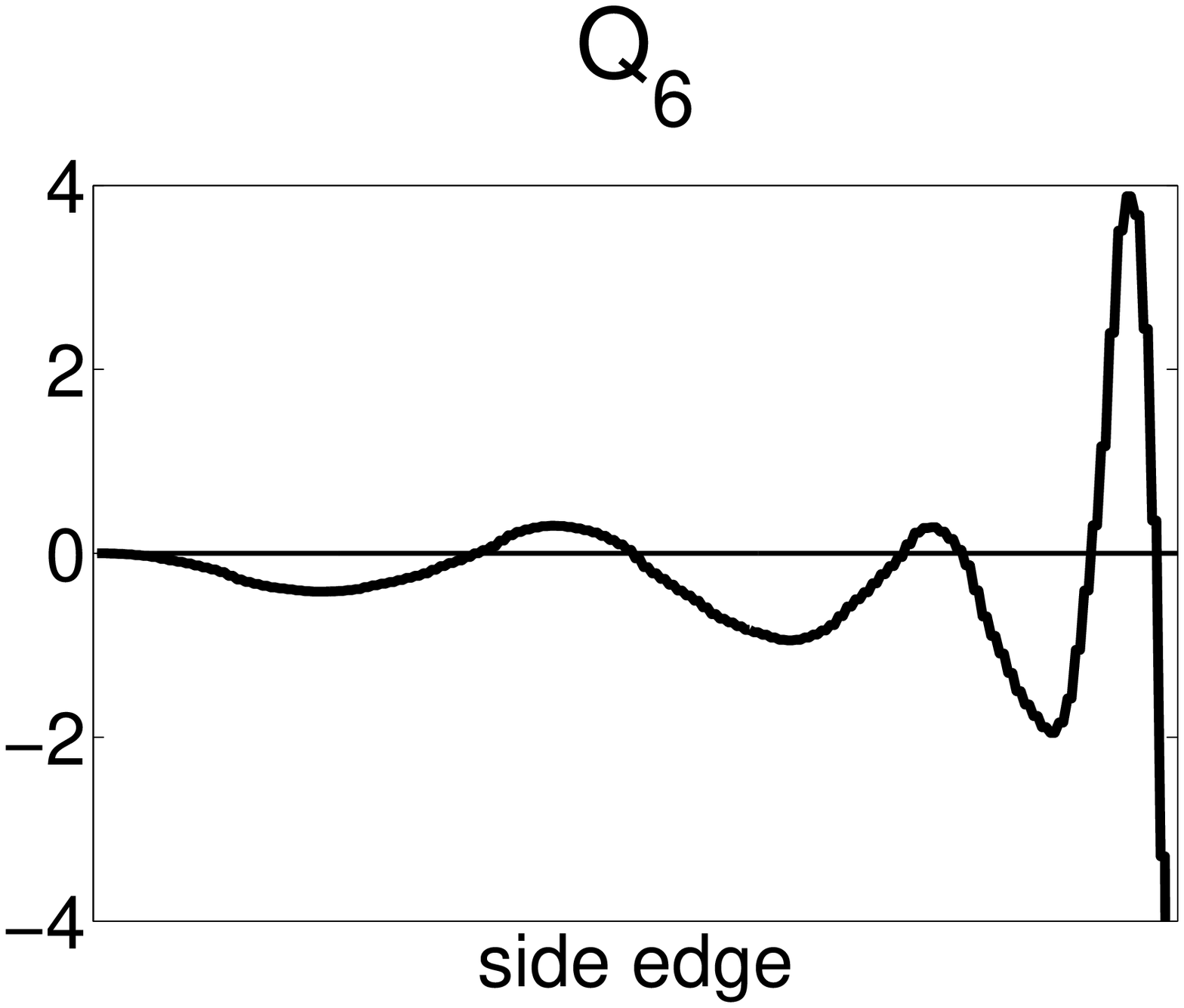}
\includegraphics[scale=.2]{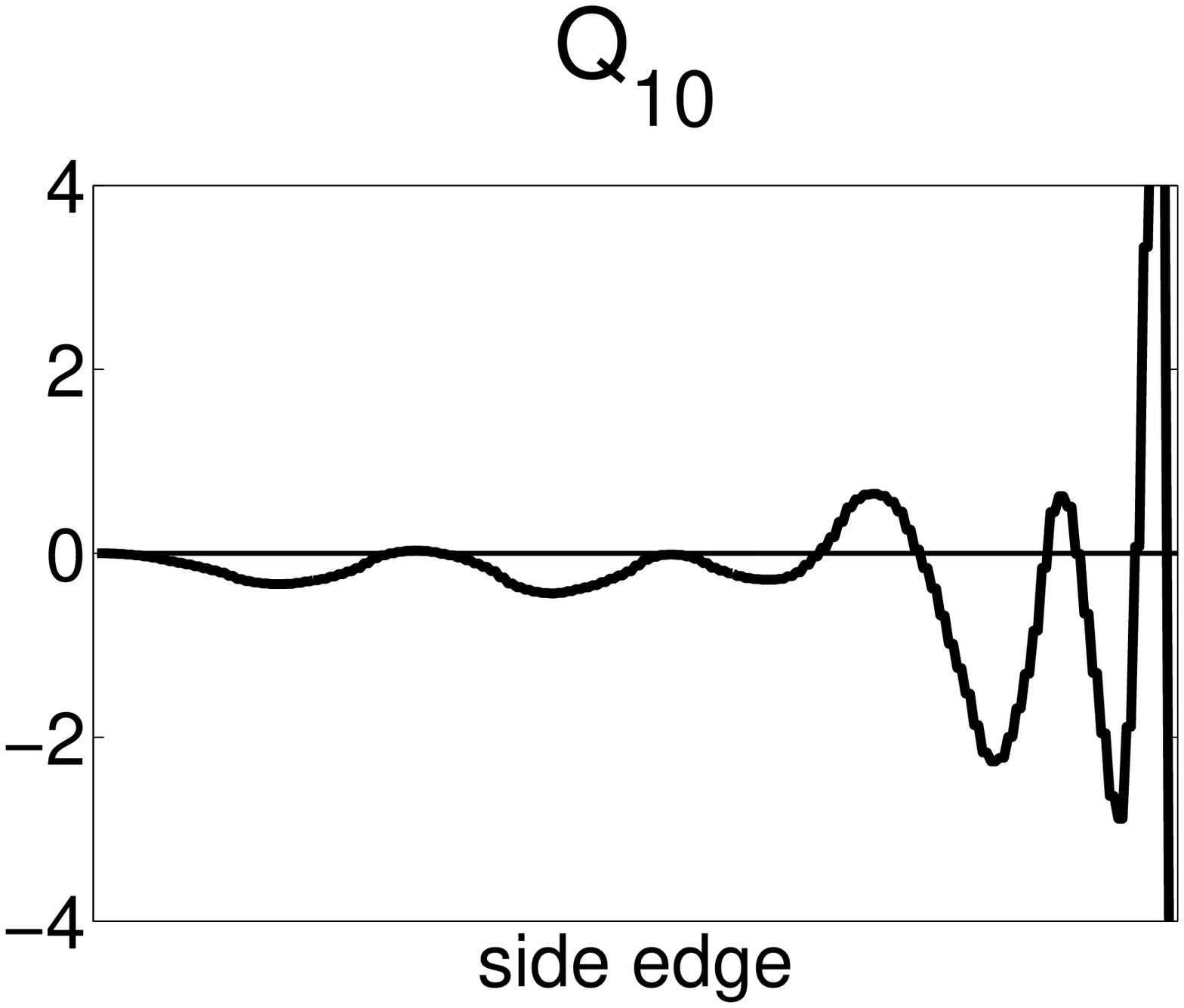}
\caption{Graph of $Q_k$, $k=1, 5, 6, 10$ on a side edge of $SG$}
\label{fig:edges}
\end{center}
\end{figure}

More generally, we have been able to numerically generate  the nodal domains of both the first $20$ antisymmetric and  first $20$ symmetric OP on $SG$.
For example, the nodal domains of a  polynomial $f$ on $SG$  are
defined as follows: Let  $$Z(f) = f^{-1}\{0\}=\{x \in SG: f (x)=0\}$$
be the zero set (or the nodal set) of $f$. Then, $SG\setminus Z(f)$
can be partitioned into finitely many connected domains $D_1, D_2,
\hdots D_{\nu_{k}}$, where $\nu_k$ depends on the degree of the
polynomial $f$. These domains $\{D_{\ell}\}_{\ell=1}^{\nu_k}$ are
\emph{the  nodal domains} of $f$. We refer to \cite{webpage} for color graphs that support the above bound on $\nu_k$.

\begin{rem}
As mentioned in the Introduction, we have generated many data for all the OP constructed here. These data as well as the codes used to generate them are available at \cite{webpage}. Certain data that are not included in the present article, deal with the dynamics of the OP at certain points of $SG$. That is, a choice of OP type and a fixed point $x\in SG$  yield a set of points $\{(p_{k}(x), p_{k+1}(x))\}$ in the plane. The passage from $k$ to $k+1$ is thought of as a dynamical system in the plane with some "attractor" toward which the points tend as $k \to \infty$. With our numerics, we had hoped to get a "snapshot" of this attractor. The computational complexity associated with our construction, limited us to to generate data for these dynamics only for $1\leq k \leq 200$. Thus, we could not make some conclusive statements about the long term behavior of these dynamics. However, we refer the interested reader to \cite{webpage} for graphs of some of these dynamics.  Our investigation in this context is  parallel to recent results for the dynamics of certain OP
related to self-similar measures on the unit interval \cite{hos10}.
\end{rem}

\section{Acknowledgment}  K.~A.~Okoudjou  was partially supported by the Alexander von Humboldt foundation, by ONR grants: N000140910324 $\&$ N000140910144, and by a RASA from the Graduate School of UMCP. He would also like to express its gratitude to the Institute for Mathematics at the 
University of Osnabr\"uck for its hospitality while part of this work was completed. R.~S.~Strichartz was supported in part  by the National Science Foundation, grant DMS-0652440. E.~K.~Tuley was  supported in part by the National Science Foundation through the Research Experiences for Undergraduates at Cornell, and by the National Science Foundation Graduate Research Fellowship, grant DGE-0707424.

\end{document}